\documentclass[a4paper,10pt,reqno]{amsart}

\usepackage[utf8]{inputenc}
\usepackage[T1]{fontenc}
\usepackage{lmodern}
\usepackage[english]{babel}
\usepackage{microtype}

\usepackage{amsmath,amssymb,amsfonts,amsthm,esint}
\usepackage{mathtools,accents}
\usepackage{mathrsfs}
\usepackage{aliascnt}
\usepackage{braket}
\usepackage{bm}

\usepackage[a4paper,margin=3cm]{geometry}
\usepackage[citecolor=blue,colorlinks]{hyperref}

\usepackage{enumerate}
\usepackage{xcolor}

\makeatletter
\g@addto@macro\@floatboxreset\centering
\makeatother

\makeatletter
\def\newaliasedtheorem#1[#2]#3{
	\newaliascnt{#1@alt}{#2}
	\newtheorem{#1}[#1@alt]{#3}
	\expandafter\newcommand\csname #1@altname\endcsname{#3}
}
\makeatother

\numberwithin{equation}{section}

\newtheoremstyle{slanted}{\topsep}{\topsep}{\slshape}{}{\bfseries}{.}{.5em}{}

\theoremstyle{plain}
\newtheorem{theorem}{Theorem}[section]
\newaliasedtheorem{proposition}[theorem]{Proposition}
\newaliasedtheorem{lemma}[theorem]{Lemma}
\newaliasedtheorem{corollary}[theorem]{Corollary}
\newaliasedtheorem{counterexample}[theorem]{Counterexample}

\theoremstyle{definition}
\newaliasedtheorem{definition}[theorem]{Definition}
\newaliasedtheorem{question}[theorem]{Question}
\newaliasedtheorem{openquestion}[theorem]{Open Question}
\newaliasedtheorem{conjecture}[theorem]{Conjecture}

\theoremstyle{remark}
\newaliasedtheorem{remark}[theorem]{Remark}
\newaliasedtheorem{example}[theorem]{Example}

\newcommand{\setN}{\mathbb{N}}

\newcommand{\setR}{\mathbb{R}}

\newcommand{\eps}{\varepsilon}

\let\altphi\phi
\let\phi\varphi
\let\varphi\altphi
\let\altphi\undefined

\newcommand{\abs}[1]{\left\lvert#1\right\rvert}
\newcommand{\norm}[1]{\left\lVert#1\right\rVert}

\newcommand{\weakto}{\rightharpoonup}

\DeclareMathOperator{\tr}{tr}
\let\div\undefined
\DeclareMathOperator{\div}{div}
\DeclareMathOperator{\Hess}{Hess}

\newcommand{\di}{\mathop{}\!\mathrm{d}}

\newcommand{\bs}{{\rm bs}}
\newcommand{\loc}{{\rm loc}}

\newcommand{\res}{\mathop{\hbox{\vrule height 7pt width .5pt depth 0pt
			\vrule height .5pt width 6pt depth 0pt}}\nolimits}

\newcommand{\Ch}{{\sf Ch}}

\DeclareMathOperator{\Lip}{Lip}
\DeclareMathOperator{\Lipb}{Lip_b}
\DeclareMathOperator{\Lipbs}{Lip_\bs}

\DeclareMathOperator{\lip}{lip}
\DeclareMathOperator{\BV}{BV}
\DeclareMathOperator{\Per}{Per}

\newcommand{\e}{{\rm e}}

\newcommand{\haus}{\mathscr{H}}
\newcommand{\Leb}{\mathscr{L}}

\newcommand{\nchi}{{\raise.3ex\hbox{\(\chi\)}}}

\newcommand{\ppi}{{\boldsymbol{\pi}}}

\newcommand{\dist}{\mathsf{d}}

\newcommand{\meas}{\mathfrak{m}}

\newcommand{\Liploc}{{\Lip_{\rm bs}}}

\newcommand{\Test}{{\rm Test}}
\newcommand{\TestV}{{\rm TestV}}

\DeclareMathOperator{\RCD}{RCD}

\newfont{\tmpf}{cmsy10 scaled 2500}

\DeclareMathOperator{\Tan}{Tan}

\def\XXint#1#2#3{{\setbox0=\hbox{$#1{#2#3}{\int}$ }
		\vcenter{\hbox{$#2#3$ }}\kern-.6\wd0}}


\begin{document}

\title[Constancy in codimension one and locality of unit normal on $\RCD(K,N)$ spaces]
{Constancy of the dimension in codimension one and locality of the unit normal on $\RCD(K,N)$ spaces}

\author{Elia Bru\'e}
	\thanks{Institute for Advanced Study, \url{elia.brue@ias.edu}.}

\author{Enrico Pasqualetto}
\thanks{Scuola Normale Superiore (Pisa), \url{enrico.pasqualetto@sns.it}.}

\author{Daniele Semola}
\thanks{Mathematical Institute, University of Oxford, \url{daniele.semola@maths.ox.ac.uk}.}

\maketitle

\begin{abstract}
		The aim of this paper is threefold. We first prove that, on $\RCD(K,N)$ spaces, the boundary measure of any set with finite perimeter is concentrated
		on the $n$-regular set $\mathcal{R}_n$, where $n\le N$ is the essential dimension of the space. 
		After, we discuss localization properties of the unit normal providing representation formulae for the perimeter measure of intersections and unions of sets with finite perimeter. Finally, we study Gauss-Green formulae for essentially bounded divergence measure vector fields, sharpening the analysis in \cite{BuffaComiMiranda19}.\\
		These tools are fundamental for the development of a regularity theory for local perimeter minimizers on $\RCD(K,N)$ spaces in \cite{MondinoSemola21}.
\end{abstract}

\tableofcontents

\section{Introduction}
In the Euclidean setting, a Borel set $E\subset\setR^n$ has finite perimeter provided its distributional derivative $D\chi_E$ is a finite Radon measure. A celebrated regularity theorem, due to De Giorgi \cite{DeGiorgi54,DeGiorgi55}, says that for any set of finite perimeter $E\subset \setR^n$, letting
\begin{equation}\label{eq:pointunit}
\mathcal{F}E:=\left\lbrace 
x\in\setR^n\, :\, \nu_E(x):=\lim_{r\to 0}\frac{D\chi_E(B_r(x))}{\abs{D\chi_E}(B_r(x))}\quad\text{exists and}\quad \abs{\nu_E(x)}=1  
\right\rbrace \, 
\end{equation}
be the \emph{reduced boundary} of $E$, the following hold:
\begin{itemize}
\item[i)] for any $x\in\mathcal{F}E$ the blow-up of $E$ at $x$ is unique and it is the half-space with interior unit normal vector $\nu_E(x)$;
\item[ii)] the representation formulae $D\chi_E=\nu_E\abs{D\chi_E}$ and $\abs{D\chi_E}=\haus^{n-1}\res\mathcal{F}E$ hold;
\item[iii)] $\mathcal{F}E$ is $(n-1)$-rectifiable.
\end{itemize}
De Giorgi's theorem motivates the use of boundaries of sets of finite perimeter as a weak notion of codimension one oriented hypersurface and is at the root of many subsequent developments of Geometric Measure Theory.
\medskip

More recently, sets of finite perimeter have been studied also on metric measure spaces, starting from \cite{Am01,Am02}. In this framework it is too optimistic to hope for a regularity theorem as strong as the Euclidean one. However, in \cite{AmbrosioBrueSemola19,BruePasqualettoSemola19} a counterpart of De Giorgi's regularity theorem has been obtained in the setting of $\RCD(K,N)$ spaces, with finite $N$, that are a class of possibly singular metric measure spaces with Ricci curvature bounded from below and dimension bounded from above, in synthetic sense.\\
We recall below two of the main results of \cite{AmbrosioBrueSemola19,BruePasqualettoSemola19}. By $\Tan_x(X,\dist,\meas,E)$ we shall denote the set of all possible blow-ups of a set of finite perimeter $E$ at a point $x$, see \autoref{def:blowupset} for the precise notion.

\begin{theorem}[Theorem 3.2 and Theorem 4.1 in \cite{BruePasqualettoSemola19}]
	Let $(X,\dist,\meas)$ be an $\RCD(K,N)$ m.m.s.\ with essential dimension $1\le n\le N$, $E\subset X$ be a set of finite perimeter. Then, for $|D\chi_E|$-a.e.\ $x\in X$, there exists $k=1,\ldots,n$ such that
	\[
	\Tan_x(X,\dist,\meas,E)=\left\lbrace (\setR^k,\dist_{eucl},c_k\Leb^k,0^k,\left\lbrace x_k>0\right\rbrace )\right\rbrace\, .
	\]
Moreover, setting 
\begin{equation}\label{eq:redintro}
\mathcal{F}_kE:= \left\lbrace x\in X\, :\, \Tan_x(X,\dist,\meas,E)= \left\lbrace (\setR^k,\dist_{eucl},c_k\Leb^k,0^k,\left\lbrace x_k>0\right\rbrace )\right\rbrace \right \rbrace\, ,
\end{equation}
it holds that $\mathcal{F}_kE$ is $(\abs{D\chi_E}, k)$-rectifiable for any $k=1,\dots, n$.
\end{theorem}

It turns out that it is also possible to reconcile the definition of set of finite perimeter via relaxation, see \autoref{def:setoffiniteperimeter}, with the distributional perspective, proving a Gauss-Green integration by parts formula for sufficiently regular vector fields.

\begin{theorem}[Theorem 2.4 in \cite{BruePasqualettoSemola19}]\label{thm:GGintro}
Let $(X,\dist,\meas)$ be an $\RCD(K,N)$ metric measure space and let $E\subset X$ be a set with finite perimeter and finite measure. Then there exists a unique vector field $\nu_E\in L^2_E(TX)$ such that $\abs{\nu_E}=1$ holds $\Per_E$-a.e.\ and
\begin{equation*}
\int _E\div v\di\meas
=-
\int \langle \mathrm{tr}_E v,\nu_E \rangle \di\Per_E\, , \quad\text{for any $v\in H^{1,2}_C(TX)\cap D(\div )$ such that $\abs{v}\in L^{\infty}(\meas)$}\, .
\end{equation*}
\end{theorem}

Above, $L^2_E(TX)$ denotes the restriction of the tangent module $L^2(TX)$ to the boundary of the set of finite perimeter $E$, see \autoref{thm:tg_mod_over_bdry} for the precise definition. The vector field $\nu_E$ plays the role of the interior unit normal in the smooth setting and the Gauss-Green formula holds when testing against vector fields in the Sobolev space $H^{1,2}_C(TX)$ (i.e.\ with $L^2$ covariant derivative, see \eqref{eq:H12}).
\medskip

This note is a further contribution to the theory of sets of finite perimeter in this setting. In particular:
\begin{itemize}
\item we will prove that reduced boundaries of sets of finite perimeter have constant dimension, positively answering to one of the questions left open in \cite{BruePasqualettoSemola19};
\item we will clarify in which sense the blow-up of a set of finite perimeter is orthogonal to its unit normal at almost every point and develop a series of useful tools suitable to treat cut and paste operations between sets of finite perimeter in this setting, by analogy with the Euclidean theory (see for instance \cite[Chapter 16]{Maggi12});
\item relying on the finite dimensionality assumption $N<\infty$, we will sharpen the Gauss-Green integration by parts formulae for essentially bounded divergence measure vector fields studied in \cite{BuffaComiMiranda19} on $\RCD(K,\infty)$ metric measure spaces.
\end{itemize}

The class of $\RCD(K,N)$ metric measure spaces includes as notable examples (pointed measured) Gromov-Hausdorff limits of smooth manifolds with uniform lower bounds on their Ricci curvature (the so called Ricci limit spaces) and Alexandrov spaces with sectional curvature bounded from below. Our results are meant to be a further development to the Geometric Measure Theory in this setting and they play a fundamental role in the study of the mean curvature and regularity of local perimeter minimizers in \cite{MondinoSemola21}. Furthermore, as the recent \cite{AntonelliFogagnoloPozzetta21,AntonelliPasqualettoPozzetta21,Antonellietal21} show, understanding classical questions of Geometric Measure Theory on non smooth $\RCD(K,N)$ spaces is relevant to investigate challenging open problems on smooth manifolds with lower Ricci curvature bounds, such as the isoperimetric problem.
\medskip

Below we outline the main achievements of this note, in comparison with the previous literature.

\subsection*{Constancy of the dimension of the reduced boundary}

In \cite{BrueSemola20} the first and third named authors proved that for any $\RCD(K,N)$ metric measure space $(X,\dist,\meas)$ there exists a natural number $1\le n\le N$, called \emph{essential dimension} of $X$, such that at $\meas$-almost every point of $X$ the unique blow-up of $(X,\dist,\meas)$ at $x$ is $\setR^n$ (with canonical metric measure structure). 
\medskip

One of the questions left open in \cite{BruePasqualettoSemola19} was the possibility of having sets of finite perimeter $E\subset X$ such that $\abs{D\chi_E}(\mathcal{F}_kE)>0$ for some $k<n$, where $\mathcal F_kE$ is as in \eqref{eq:redintro} and $n$ denotes the essential dimension of $(X,\dist,\meas)$ as above. Our first main result is a negative answer to this question.

\begin{theorem}[Cf.\ with \autoref{cor:repr} below]\label{thm:cdimintro}
Let $(X,\dist,\meas)$ be an $\RCD(K,N)$ metric measure space for $K\in\setR$ and $1\le N<\infty$. Let $1\le n\le N$ be its essential dimension. Then for any set of finite perimeter $E\subset X$ it holds that $\abs{D\chi_E}(\mathcal{F}_kE)=0$ for any $k\neq n$.
\end{theorem}

The proof of \autoref{thm:cdimintro} builds on \cite{Deng20}, where H\"older continuity of tangent cones in the interior of geodesics has been proved on $\RCD(K,N)$ spaces (see also the previous \cite{ColdingNaber12}), and on a recent characterization of $\BV$ functions via test plans concentrated on geodesics \cite{NobiliPasqualettoSchultz21}. Actually, \autoref{thm:cdimintro} will be a corollary of a more general result dealing with arbitrary functions with bounded variation, see \autoref{thm:tv_on_reg_set}. Some consequences at the level of the local dimension of the tangent module $L^2_E(TX)$ will also be addressed, see \autoref{thm:constdim tanmod}.

Notice that it is currently unknown whether on an $\RCD(K,N)$ space, or a collapsed Ricci limit space, with essential dimension $n$ there could be points where all tangent cones are Euclidean of dimension $k<n$ or not. A corollary of \autoref{thm:cdimintro} is that there cannot be collapse relative to the ambient along boundaries of sets of finite perimeter, see \autoref{cor:noncoll} below for a precise statement.

\subsection*{Pointwise behaviour of the unit normal and operations with sets of finite perimeter}
The pointwise characterization \eqref{eq:pointunit} is a key tool for the proof of De Giorgi's theorem. Moreover, the fact that blow-ups at reduced boundary points are half-spaces orthogonal to the unit normal is of fundamental importance for the sake of many applications, for instance to analyze the behaviour of the unit normal with respect to natural cut and paste operations between sets of finite perimeter, as in \cite[Chapter 16]{Maggi12}.

In \cite{AmbrosioBrueSemola19,BruePasqualettoSemola19} a new set of ideas was needed to develop the regularity theory for sets of finite perimeter, as it was necessary to avoid \eqref{eq:pointunit} and the use of Besicovitch differentiation theorem. However, after \autoref{thm:GGintro} it is natural to investigate if blow-ups of sets of finite perimeter are orthogonal to their unit normal, in some sense.\\
The second main result of this paper is a positive answer to this question.

\begin{definition}[See \autoref{def:goodcord} below]
Let $(X,\dist,\meas)$ be an $\RCD(K,N)$ metric measure space for some $K\in\setR$ and $1\le N<\infty$ with essential dimension $1\le n\le N$ and let $E\subset X$ be a set of finite perimeter. Then, for any $x\in \mathcal{F}_n E$, any $n$-tuple of harmonic functions $(u_i):B_{r_x}(x)\to\setR^n$ satisfying the following properties is called a system of \emph{good coordinates} for $E$ at $x$.

\begin{itemize}

\item[(i)] For any $i,j\in \{ 1,\dots, n \}$,
\begin{equation*}
	\lim_{r\to 0}\fint_{B_r(x)}\abs{\nabla u_i\cdot\nabla u_j-\delta_{ij}}\di\meas=\lim_{r\to 0}\fint_{B_r(x)}\abs{\nabla u_i\cdot\nabla u_j-\delta_{ij}}\di\abs{D\chi_E}=0\, .
\end{equation*}

\item[(ii)] For any $i\in\{1,\dots, n\}$ the following limits exist:
\begin{equation*}
	\nu_i(x):=\lim_{r\to 0}\fint_{B_r(x)}\nu_E\cdot\nabla u_i\di\abs{D\chi_E},
\end{equation*}
\begin{equation*}
	\lim_{r\to 0}\fint_{B_r(x)}\abs{\nu_i(x)-\nu_E\cdot\nabla u_i}\di\abs{D\chi_E}=0 \, .
\end{equation*}
\end{itemize}
\end{definition}

In \autoref{prop:harmcoord} we are going to prove that good coordinates exist at almost every point with respect to the perimeter measure and that the vector $\nu\in\setR^n$ defined above is of unit length. Moreover in \autoref{prop:blowupsandharmcoord} we characterize the blow-up of a set of finite perimeter as the half-space orthogonal to the vector $\nu$ constructed by means of the good coordinates, providing a counterpart of the classical Euclidean result.\\
Building on top of these tools, in \autoref{thm:cutandpaste} we consider the behaviour of the unit normal of unions and intersections of sets of finite perimeter. With respect to previous results in the literature (see for instance \cite{AntonelliPasqualettoPozzetta21}), the main contribution here is the study of the set where the two sets have \emph{mutually tangent} boundaries.

\subsection*{Gauss-Green formulae for essentially bounded divergence measure vector fields}

In many situations in (non smooth) Geometric Analysis it is necessary to deal with functions that are not smooth but satisfy certain \emph{second order} bounds in a weak sense. A typical example are distance functions on manifolds with lower Ricci curvature bounds: they are not globally differentiable, in general, but the Laplacian comparison holds globally in the sense of distributions. There are usually two possibilities to deal with these functions: 
\begin{itemize}
\item argue by ad-hoc regularizations, preserving the good bounds while gaining smoothness and rely on classical tools, then pass to the limit; 
\item prove that classical tools for smooth functions and domains (such as Gauss-Green formulae and divergence theorem) hold also under weaker regularity assumptions.
\end{itemize}

Another main result of this note, going in the second direction hinted above, is the extension of \autoref{thm:GGintro} to the case of \emph{essentially bounded divergence measure} vector fields.

\begin{definition}
Let $(X,\dist,\meas)$ be an $\RCD(K,N)$ metric measure space. We say that a vector field $V\in L^{\infty}(TX)$ is an essentially bounded divergence measure vector field if its distributional divergence is a finite Borel measure, that is if $\div V$ is a finite Borel measure such that, for any Lipschitz function with compact support $g:X\to\setR$, it holds
\[
\int_X g\di \div V=-\int_X \nabla g\cdot V\di\meas\, .
\]
\end{definition}

The introduction of this class of vector fields in the Euclidean setting dates back to \cite{Anzellotti83}.
For the sake of the $\RCD$ theory, the key remark is that a large family of essentially bounded divergence measure vector fields is given by gradients of distance functions, thanks to the Laplacian comparison theorem \cite{Gigli15}.
\medskip

Gauss-Green integration by parts formulae for sets of finite perimeter and vector fields with such low regularity in the Euclidean setting have been studied in \cite{ChenTorresZiemer09,ComiPayne20}. Later on, in \cite{BuffaComiMiranda19} the theory has been partially extended to locally compact $\RCD(K,\infty)$ metric measure spaces (see also the recent \cite{Braun21}). Here, fully exploiting the finite dimensionality assumption $N<\infty$ and the regularity theory for sets of finite perimeter, we achieve a quite complete extension of the Euclidean results, sharpening those in \cite{BuffaComiMiranda19}.

\begin{theorem}[See \autoref{thm:GaussGreenEssDivVect} below]\label{thmLGGweakintro}
	Let $(X,\dist,\meas)$ be an $\RCD(K,N)$ metric measure space for some $K\in\setR$ and $1\le N<\infty$. Let $E\subset X$ be a set of finite perimeter and let $V\in\mathcal{DM}^{\infty}(X)$. Then we have the Gauss-Green integration by parts formulae: for any function $\phi\in \Lip_c(X)$ it holds
	\[\begin{split}
		\int_{E^{(1)}}\phi \div V+\int_E\nabla\phi\cdot V\di\meas&=-\int_{\mathcal{F}E}\phi\left(V\cdot\nu_E\right)_{\mathrm{int}}\di\abs{D\chi_E} \, ,\\
		\int_{E^{(1)}\cup\mathcal{F}E}\phi\div V+\int_E\nabla\phi\cdot V\di\meas&=-\int_{\mathcal{F}E}\phi\left(V\cdot\nu_E\right)_{\mathrm{ext}}\di\abs{D\chi_E}\, ,
	\end{split}\]
	where $(V\cdot \nu_E)_{_{\mathrm{int}}}$ and $(V\cdot \nu_E)_{_{\mathrm{ext}}}$ belong to $L^{\infty}(\mathcal{F}E,\abs{D\chi_E})$ and satisfy
   \[\begin{split}
	\norm{\left(V\cdot\nu_E\right)_{\mathrm{int}}}_{L^{\infty}(\mathcal{F}E,\abs{D\chi_E})}&\le \norm{V}_{L^{\infty}(E,\meas)} \, ,\\
	\norm{\left(V\cdot\nu_E\right)_{\mathrm{ext}}}_{L^{\infty}(\mathcal{F}E,\abs{D\chi_E})}&\le \norm{V}_{L^{\infty}(X\setminus E,\meas)}\, .
   \end{split}\]
\end{theorem}

We refer to \autoref{sec5} for the precise definitions of the various terms appearing in the Gauss-Green formulae above. We just remark that $(V\cdot\nu_E)_{\mathrm{int}}$ and $(V\cdot\nu_E)_{\mathrm{ext}}$ play the role of the interior and exterior normal traces of the vector field $V$ on the boundary of the set of finite perimeter $E$. With respect to the case of regular $H^{1,2}_C(TX)$ vector fields, these traces might be different, as it happens in simple examples where the divergence of $V$ has a singular part on $\mathcal{F}E$.

The precise understanding of the normal traces in \autoref{thmLGGweakintro} allows us also to prove that they behave well under the natural cut and paste operations (see \autoref{prop:cutandpasteweaker}), in analogy with the Euclidean theory (see for instance \cite[Chapter 3]{Comi20}).

%

\subsection*{Acknowledgements}
The first named author is supported by the Giorgio and Elena Petronio Fellowship at the Institute for
Advanced Study. The second named author is supported by the Balzan project led by Luigi Ambrosio. The third named author is supported by the European Research Council (ERC), under the European’s Union Horizon 2020 research and innovation programme, via the ERC Starting Grant “CURVATURE”, grant agreement No. 802689. The authors are grateful to Gioacchino Antonelli, Camillo Brena and Nicola Gigli for useful comments on the preliminary version of this manuscript.

\section{Preliminaries}\label{sec:preliminaries}
\subsection{Basic calculus tools}

Throughout this paper a \textit{metric measure space} is a triple $(X,\dist,\meas)$, where $(X,\dist)$ is a complete and separable metric space and $\meas$ is a nonnegative Borel measure on $X$ finite on bounded sets.

We will denote by $B_r(x)=\{\dist(\cdot,x)<r\}$ and $\bar{B}_r(x)=\{\dist(\cdot,x)\leq r\}$ the open and closed balls respectively, by $\Lip(X,\dist)$ (resp.\ $\Lipb(X,\dist)$, $\Lip_c(X,\dist)$, $\Lip_{\rm bs}(X,\dist)$, $\Lip_{\text{loc}}(X,\dist)$) the space of Lipschitz
(resp.\ bounded Lipschitz, compactly supported Lipschitz, Lipschitz with bounded support, Lipschitz on bounded sets) functions and for any $f\in\Lip(X,\dist)$ we shall denote its slope by
\begin{equation*}
\lip f(x):=\limsup_{y\to x}\frac{\abs{f(x)-f(y)}}{\dist(x,y)}\quad\text{ for every accumulation point }x\in X
\end{equation*}
and \(\lip f(x):=0\) elsewhere.
We shall use the standard notation $L^p(X,\meas)=L^p(\meas)$ for the $L^p$ spaces and $\Leb^n$ for the $n$-dimensional Lebesgue measure on $\setR^n$. We shall denote by $\omega_n$ the Lebesgue measure of the unit ball in $\setR^n$. If $f\in L^1_{\rm loc}(X,\meas)$ and $U\subset X$ is such that $0<\meas(U)< \infty$, then $\fint_Uf\di \meas$ denotes the average of $f$ over $U$.

The Cheeger energy $\Ch:L^2(X,\meas)\to[0,+\infty]$ is the convex and lower semicontinuous functional defined through
\begin{equation}\label{eq:cheeger}
\Ch(f):= \inf\left\lbrace\liminf_{n\to\infty}\int_X(\lip f_n)^2\di \meas\, :\, \quad f_n\in\Lipb(X)\cap L^2(X,\meas),\ \norm{f_n-f}_2\to 0 \right\rbrace
\end{equation}
and its finiteness domain will be denoted by $H^{1,2}(X,\dist,\meas)$, sometimes we write $H^{1,2}(X)$ omitting the dependence on $\dist$ and $\meas$ when it is clear from the context. Looking at the optimal approximating sequence in \eqref{eq:cheeger}, it is possible to identify a canonical object $\abs{\nabla f}$, called minimal relaxed slope, providing the integral representation
\begin{equation*}
\Ch(f):= \int_X\abs{\nabla f}^2\di \meas\qquad\forall f\in H^{1,2}(X,\dist,\meas)\, .
\end{equation*}

Any metric measure space such that $\Ch$ is a quadratic form is said to be \textit{infinitesimally Hilbertian} \cite{Gigli15}. Let us recall from \cite{AmbrosioGigliSavare14a,AmbrosioGigliSavare14b,Gigli13} that, under this assumption, the function 
\begin{equation*}
\nabla f_1\cdot\nabla f_2:= \lim_{\eps\to 0}\frac{\abs{\nabla(f_1+\eps f_2)}^2-\abs{\nabla f_1}^2}{2\eps}
\end{equation*}
defines a symmetric bilinear form on $H^{1,2}(X,\dist,\meas)\times H^{1,2}(X,\dist,\meas)$ with values into $L^1(X,\meas)$.

It is possible to define a Laplacian operator $\Delta:\mathcal{D}(\Delta)\subset L^{2}(X,\meas)\to L^2(X,\meas)$ in the following way. We let $\mathcal{D}(\Delta)$ be the set of those $f\in H^{1,2}(X,\dist,\meas)$ such that, for some $h\in L^2(X,\meas)$, one has
\begin{equation}\label{eq:amb1}
\int_X \nabla f\cdot\nabla g\, \di\meas=-\int_X hg\,\di\meas\qquad\forall g\in H^{1,2}(X,\dist,\meas) 
\end{equation} 
and, in that case, we put $\Delta f=h$. 
It is easy to check that the definition is well-posed and that the Laplacian is linear (because $\Ch$ is a quadratic form). 

The heat flow $P_t$ is defined as the $L^2(X,\meas)$-gradient flow of $\frac{1}{2}\Ch$. Its existence and uniqueness
follow from the Komura-Brezis theory. It can be equivalently characterized by saying that for any
$u\in L^2(X,\meas)$ the curve $t\mapsto P_tu\in L^2(X,\meas)$ is locally absolutely continuous in $(0,+\infty)$ and satisfies
\begin{equation*}
\frac{\di}{\di t}P_tu=\Delta P_tu \quad\text{for }\Leb^1\text{-a.e.\ }t\in(0,+\infty)\, ,\qquad
\lim_{t\downarrow 0}P_tu=u\quad\text{in $L^2(X,\meas)$\, .}
\end{equation*}
Under the infinitesimal Hilbertianity assumption the heat flow provides a linear, continuous and self-adjoint contraction semigroup in $L^2(X,\meas)$. Moreover $P_t$ extends to a linear, continuous and mass preserving operator,
still denoted by $P_t$, in all the $L^p$ spaces for $1\le p < +\infty$.

\subsection{About normed modules}
Here we discuss some basic concepts in the theory of normed modules (see \cite{Gigli18} after \cite{Weaver00}),
whose aim is to provide a solid functional-analytic framework where to give a notion of `vector field'.
The original approach to this matter (where normed modules were defined over a metric
space endowed with a Borel measure) is not sufficient for our purposes, since we would like to work also
with vector fields defined capacity-a.e.\ as in \cite{DebinGigliPasqualetto21}. Accordingly, we will propose
in \autoref{def:normed_mod} below a notion of normed module which unifies the two theories studied in
\cite{Gigli18} and \cite{DebinGigliPasqualetto21}.
\medskip

Given a metric space \((X,\dist)\), we denote by \(\mathscr B(X)\)
the Borel \(\sigma\)-algebra on \(X\). We say that an outer measure
\(\mu\) on \(X\) is \emph{boundedly finite} provided \(\mu(E)<+\infty\)
for every \(E\in\mathscr B(X)\) bounded. Moreover, we say that \(\mu\)
is \emph{submodular} provided
\[
\mu(E\cup F)+\mu(E\cap F)\leq\mu(E)+\mu(F)
\quad\text{ for every }E,F\in\mathscr B(X)\, .
\]
For a non-negative Borel function \(f\colon X\to[0,+\infty]\),
the integral of \(f\) with respect to \(\mu\) on a set
\(E\in\mathscr B(X)\) can be defined via Cavalieri's formula as
\[
\int_E f\,\di\mu\coloneqq\int_0^{+\infty}\mu(\{\chi_E f>t\})\,\di t\, .
\]
It holds that the integral operator \(f\mapsto\int_X f\,\di\mu\)
is subadditive, i.e.\
\[
\int_X f+g\,\di\mu\leq\int_X f\,\di\mu+\int_X g\,\di\mu
\quad\text{ for every }f,g\colon X\to[0,+\infty]\text{ Borel}\, ,
\]
if and only if \(\mu\) is submodular; for a proof we refer to \cite[Chapter 6]{Denneberg94} (see also \cite[Theorem 1.5]{DebinGigliPasqualetto21}).
We will mostly consider the cases where
\begin{equation}\label{eq:def_good_outer_meas}
\mu\quad\text{ is a boundedly finite, submodular outer measure on }(X,\dist)\, . 
\end{equation}
We will actually deal with two classes of outer measures satisfying
\eqref{eq:def_good_outer_meas}:
\begin{itemize}
\item[\(\rm i)\)] Boundedly finite Borel measures \(\meas\) on
\((X,\dist)\); by definition, \(\meas\) is defined just on
\(\mathscr B(X)\), but we tacitly adopt the same notation
for the induced outer measure obtained via Carath\'{e}odory
construction, namely we set
\[
\meas(S)\coloneqq\inf\big\{\meas(E)\,:\,E\in\mathscr B(X),
\,S\subset E\big\}\quad\text{ for every }S\subset X\, .
\]
\item[\(\rm ii)\)] The Sobolev \(2\)-capacity \(\rm Cap\) on a metric
measure space \((X,\dist,\meas)\), which is defined as
\[
{\rm Cap}(S)\coloneqq\inf\Big\{\|f\|_{H^{1,2}(X)}^2\;\Big|\;f\in
H^{1,2}(X),\,f\geq 1\;\meas\text{-a.e.\ on a neighbourhood of }S\Big\}
\]
for every subset \(S\subset X\). We know that \(\rm Cap\)
satisfies \eqref{eq:def_good_outer_meas} and \(\meas\ll{\rm Cap}\),
cf.\ \cite{DebinGigliPasqualetto21}.
\end{itemize}
Given any \(\mu\) as in \eqref{eq:def_good_outer_meas}, we define
\(L^0(\mu)\) as the space of all the equivalence classes, up to
\(\mu\)-a.e.\ equality, of Borel functions \(f\colon X\to\setR\).
We define a distance on \(L^0(\mu)\): let us fix a sequence \((A_n)_n\)
of bounded open subsets of \(X\) with \(A_n\subset A_{n+1}\)
for every \(n\in\setN\) and such that for any \(B\subset X\) bounded
it holds \(B\subset A_n\) for some \(n\in\setN\) (thus
\(X=\bigcup_n A_n\)); then we set
\[
\dist_{L^0(\mu)}(f,g)\coloneqq\sum_{n\in\setN}
\frac{1}{2^n\max\{\mu(A_n),1\}}\int_{A_n}\min\{|f-g|,1\}\,\di\mu
\quad\text{ for every }f,g\in L^0(\mu)\, .
\]
The fact that \(\dist_{L^0(\mu)}\) is a distance (more precisely,
that it verifies the triangle inequality) follows from the
submodularity of \(\mu\). By inspecting the proof of
\cite[Proposition 1.10]{DebinGigliPasqualetto21}, which was written just for the case
\(\mu={\rm Cap}\), one can realize that
\(\lim_i\dist_{L^0(\mu)}(f_i,f)=0\) if and only if
\begin{equation}\label{eq:equiv_L0_conv}
\lim_{i\to\infty}\mu\big(B\cap\{|f_i-f|>\varepsilon\}\big)=0\quad
\text{ for every }\varepsilon>0\text{ and }B\subset X\text{ bounded\, .}
\end{equation}
In particular, as in \cite[Proposition 1.12]{DebinGigliPasqualetto21},
\eqref{eq:equiv_L0_conv} implies the existence of a subsequence
\((i_j)_j\) such that \(f_{i_j}(x)\to f(x)\) holds for
\(\mu\)-a.e.\ \(x\in X\). The converse implication (which
is verified, for instance, when \(\mu\) is a Borel measure)
in general might fail, see e.g.\ \cite[Remark 1.13]{DebinGigliPasqualetto21}.
\medskip

We now introduce the notion of \(L^0(\mu)\)-normed \(L^0(\mu)\)-module
when \(\mu\) is chosen as in \eqref{eq:def_good_outer_meas}. Before
doing so, we point out that \(L^0(\mu)\) is a topological vector space
and a topological ring when endowed with the natural pointwise
operations and the complete distance \(\dist_{L^0(\mu)}\).
\begin{definition}[\(L^0(\mu)\)-normed \(L^0(\mu)\)-module]\label{def:normed_mod}
Let \((X,\dist)\) be a complete separable metric space. Let \(\mu\)
be as in \eqref{eq:def_good_outer_meas}. Let \(\mathscr M\) be an
algebraic module over the commutative ring \(L^0(\mu)\).
A \emph{pointwise norm} on \(\mathscr M\) is a mapping
\(|\cdot|\colon\mathscr M\to L^0(\mu)\) which satisfies
in the \(\mu\)-a.e.\ sense
\[\begin{split}
|v|\geq 0&\quad\text{ for every }v\in\mathscr M,\text{ with equality
if and only if }v=0,\\
|v+w|\leq|v|+|w|&\quad\text{ for every }v,w\in\mathscr M,\\
|f\cdot v|=|f||v|&\quad\text{ for every }f\in L^0(\mu)\text{ and }
v\in\mathscr M.
\end{split}\]
We say that \((\mathscr M,|\cdot|)\), or just \(\mathscr M\),
is an \emph{\(L^0(\mu)\)-normed \(L^0(\mu)\)-module} provided
the distance
\[
\dist_{\mathscr M}(v,w)\coloneqq\dist_{L^0(\mu)}(|v-w|,0)
\quad\text{ for every }v,w\in\mathscr M
\]
is complete.
\end{definition}
The above definition coincides with the ones in
\cite[Definition 1.3.1]{Gigli18} and \cite[Definition 2.1]{DebinGigliPasqualetto21}
when \(\mu\) is a Borel measure and \(\mu={\rm Cap}\), respectively.
\medskip

We define the restriction \(\mathscr M|_E\) of an
\(L^0(\mu)\)-normed \(L^0(\mu)\)-module \(\mathscr M\)
to \(E\in\mathscr B(X)\) as
\[
\mathscr M_E\coloneqq\big\{\chi_E\cdot v\,:\,v\in\mathscr M\big\}.
\]
It holds that \(\mathscr M|_E\) inherits from \(\mathscr M\)
the structure of \(L^0(\mu)\)-normed \(L^0(\mu)\)-module.
Moreover, we say that an \(L^0(\mu)\)-normed \(L^0(\mu)\)-module
\(\mathscr H\) is a \emph{Hilbert module} provided
\[
|v+w|^2+|v-w|^2=2|v|^2+2|w|^2\;\;\;\mu\text{-a.e.}
\quad\text{ for every }v,w\in\mathscr H,
\]
which we shall refer to as the \emph{pointwise parallelogram rule}.
On Hilbert modules, the formula
\[
\langle v,w\rangle\coloneqq\frac{|v+w|^2-|v|^2-|w|^2}{2}
\quad\text{ for every }v,w\in\mathscr H
\]
defines an \(L^0(\mu)\)-bilinear and continuous mapping
\(\langle\cdot,\cdot\rangle\colon\mathscr H\times\mathscr H\to L^0(\mu)\).
\begin{definition}[Upper dimension bound]\label{def:local_dim}
Let \((X,\dist)\) be a complete separable metric space.
Let \(\mu\) be as in \eqref{eq:def_good_outer_meas}. Let
\(\mathscr M\) be an \(L^0(\mu)\)-normed \(L^0(\mu)\)-module.
Fix \(E\in\mathscr B(X)\) with \(\mu(E)>0\). Then we give
the following definitions:
\begin{itemize}
\item[\(\rm i)\)] A family \(\mathcal S\subset\mathscr M\)
is said to \emph{generate} \(\mathscr M\) on \(E\) provided
\[
\bigg\{\sum_{i=1}^n(\chi_E f_i)\cdot v_i\;\bigg|\;
n\in\setN,\,(f_i)_{i=1}^n\subset L^0(\mu),\,
(v_i)_{i=1}^n\subset\mathcal S\bigg\}\quad\text{ is dense in }
\mathscr M|_E\, .
\]
When \(E=X\), we just say that \(\mathcal S\) generates \(\mathscr M\).
\item[\(\rm ii)\)] Some elements \(v_1,\ldots,v_n\in\mathscr M\)
are said to be \emph{linearly independent} on \(E\) provided
\[
(f_i)_{i=1}^n\subset L^0(\mu),\;\sum_{i=1}^n(\chi_E f_i)\cdot v_i=0
\quad\Longrightarrow\quad \chi_E f_i=0\;\;\mu\text{-a.e.\ for all }
i=1,\ldots,n\, .
\]
When this is not verified, we say that \(v_1,\ldots,v_n\)
are \emph{linearly dependent} on \(E\).
\item[\(\rm iii)\)] We say that the local dimension of \(\mathscr M\)
on \(E\) \emph{does not exceed} \(n\in\setN\) provided there exists a
generating subset \(\mathcal S\) of \(\mathscr M\) having the
following property: if \(v_1,\ldots,v_{n+1}\in\mathcal S\)
and \(B\in\mathscr B(X)\) satisfies \(B\subset E\) and \(\mu(B)>0\), then
\(v_1,\ldots,v_{n+1}\) are linearly dependent on \(B\).
\end{itemize}
\end{definition}
On Hilbert normed modules, the linear independence can be checked
in the following way:
\begin{lemma}\label{lemma:matrixlin}
Let \((X,\dist)\) be a complete separable metric space. Let \(\mu\) be
as in \eqref{eq:def_good_outer_meas}. Let \(\mathscr H\) be a Hilbert
\(L^0(\mu)\)-normed \(L^0(\mu)\)-module. Let \(E\in\mathscr B(X)\)
satisfy \(\mu(E)>0\). Then \(v_1,\ldots,v_n\in\mathscr H\) are linearly
independent on \(E\) if and only if the matrix
\[
A_{ij}\coloneqq\langle v_i,v_j\rangle
\]
is invertible \(\mu\)-almost everywhere on \(E\).
\end{lemma}
\begin{proof}
Let us notice that in general $A_{ij}$ is a symmetric nonnegative matrix $\mu$-a.e.\ on $E$.

Let us consider functions $f_1,\dots ,f_n\in L^0(\mu)$ and compute
\begin{equation}\label{eq:normcom}
\abs{\sum_{i=1}^nf_iv_i}^2=\sum_{i,j=1}^nA_{ij}f_if_j\, ,\quad\text{$\mu$-a.e.\ on $E$}\, .
\end{equation}
If we assume that $A_{ij}$ is invertible $\mu$-a.e.\ on $E$ and $\sum_{i=1}^nf_iv_i=0$ $\mu$-a.e.\ on $E$,
then $A_{ij}$ is positive definite $\mu$-a.e.\ on $E$ and from \eqref{eq:normcom} it follows that 
$f_i=0$ holds $\mu$-a.e.\ on $E$ for any $i=1,\dots, n$. Hence $v_1,\dots,v_n$ are linearly independent on $E$.

Vice versa, if $v_1,\dots,v_n$ are linearly independent on $E$, and we suppose that $A_{ij}$ is singular on $B\subset E$ such that $\mu(B)>0$,
then we can find functions $g_1,\dots ,g_n\in L^0(\mu)$ such that $\sum_{i=1}^ng_i^2>0$ $\mu$-a.e.\ on $B$, while $\sum_{i,j=1}^nA_{ij}g_ig_j=0$
$\mu$-a.e.\ on $B$. Hence $\sum_{i=1}^n g_iv_i=0$ $\mu$-a.e.\ on $B$, a contradiction with the linear independence of $v_1,\dots,v_n$ on $E$.
\end{proof}

In the case where \(\mu\) is a Borel measure, the above notions
of generating set and linear (in)dependence are consistent with
those introduced in \cite[Definitions 1.4.1 and 1.4.2]{Gigli18}.
In that case, they lead to a natural notion of local dimension:
one says that \(\mathscr M\) has \emph{local dimension} equal to \(n\)
on \(E\) if there exists a \emph{local basis}
\(v_1,\ldots,v_n\in\mathscr M\) for \(\mathscr M\) on \(E\),
i.e.\ \(v_1,\ldots,v_n\) are linearly independent on \(E\)
and \(\{v_1,\ldots,v_n\}\) generates \(\mathscr M\) on \(E\);
cf.\ \cite[Definition 1.4.3]{Gigli18}.
This notion of local dimension is well-posed, because any two
local bases on a given Borel set must have the same cardinality,
see \cite[Proposition 1.4.4]{Gigli18}. Once the concept of local
dimension is established, it is possible to get the \emph{dimensional
decomposition} of \(\mathscr M\) \cite[Proposition 1.4.5]{Gigli18}:
there exists a Borel partition
\(\{D_n(\mathscr M)\}_{n\in\setN\cup\{\infty\}}\) of \(X\),
unique up to \(\mu\)-a.e.\ equality,
having the property that \(\mathscr M\) has local dimension equal
to \(n\) on \(D_n(\mathscr M)\) for every \(n\in\setN\), while
\(\mathscr M\) is not finitely-generated on any Borel subset
of \(D_\infty(\mathscr M)\) having positive \(\mu\)-measure.\\
The following consistency check is in order. For simplicity,
we just focus on Hilbert normed modules as this is sufficient
for our purposes, but the Hilbertianity assumption could be dropped.
\begin{lemma}\label{lem:equiv_dim}
Let \((X,\dist,\meas)\) be a metric measure space. Let \(\mathscr H\)
be a Hilbert \(L^0(\meas)\)-normed \(L^0(\meas)\)-module.
Then the local dimension of \(\mathscr H\) on \(E\in\mathscr B(X)\) does not exceed
\(n\in\setN\) if and only if
\begin{equation}\label{eq:equiv_dim}
\meas(E\cap D_k(\mathscr H))=0\quad\text{ for every }
k\in\setN\cup\{\infty\}\text{ with }k>n\, .
\end{equation}
\end{lemma}
\begin{proof}
To prove necessity, we argue by contradiction: suppose to have a Borel set
\(B\subset E\) with \(\meas(B)>0\) such that \(\mathscr H\)
has dimension at least \(n+1\) on \(B\). Then we can find
\(w_1,\ldots,w_{n+1}\in\mathscr H\) which are orthonormal
on \(B\), namely \(\langle w_i,w_j\rangle=\delta_{ij}\) holds
\(\meas\)-a.e.\ on \(B\) for all \(i,j=1,\ldots,n+1\). Fix a family \(\mathcal S\subset\mathscr H\)
which realizes the upper local dimension bound of \(\mathscr H\) on \(E\). Without loss of generality,
we may assume that \(\mathcal S\) is a linear subspace of \(\mathscr H\). Thanks to Egorov theorem,
we can find a compact set \(K\subset B\) with \(\meas(K)>0\) and elements
\(v_1,\ldots,v_{n+1}\in\mathcal S\) such that
\(|v_i-w_i|\leq\varepsilon\) holds \(\meas\)-a.e.\ on \(K\) for every
\(i=1,\ldots,n+1\), where \(\varepsilon>0\) is chosen so that
\((1-2\eps - \eps^2)-n(\varepsilon^2+2\varepsilon)\geq\frac{1}{2}\).
Notice that
\begin{equation}\label{eq:lin_ind_aux}
\big|\langle v_i,v_j\rangle-\langle w_i,w_j\rangle\big|
\leq|v_i-w_i||v_j|+|w_i||v_j-w_j|\leq
\varepsilon\big(|w_j|+\varepsilon\big)+|w_i|\varepsilon
=\varepsilon^2+2\varepsilon\, ,
\end{equation}
holds \(\meas\)-a.e.\ on \(K\) for every \(i,j=1,\ldots,n+1\). Given any \(f_1,\ldots,f_{n+1}\in L^0(\meas)\), it holds that
\[\begin{split}
\bigg|\sum_{i=1}^{n+1}f_i v_i\bigg|^2
&=\sum_{i=1}^{n+1}f_i^2|v_i|^2+\sum_{i\neq j}f_i f_j\langle v_i,v_j\rangle
\overset{\eqref{eq:lin_ind_aux}}\geq
(1-2\eps - \eps^2)\sum_{i=1}^{n+1}f_i^2-(\varepsilon^2+2\varepsilon)
\sum_{i\neq j}|f_i||f_j|\\
&\geq(1-2\eps - \eps^2)\sum_{i=1}^{n+1}|v_i|^2
-\frac{1}{2}(\varepsilon^2+2\varepsilon)\sum_{i\neq j}\big(f_i^2+f_j^2\big)\\
&=\big[(1-2\eps - \eps^2)-n(\varepsilon^2+2\varepsilon)\big]
\sum_{i=1}^{n+1}f_i^2\geq\frac{1}{2}\sum_{i=1}^{n+1}f_i^2\, ,
\end{split}\]
\(\meas\)-a.e.\ on \(K\). In particular, if \(\chi_K\sum_{i=1}^{n+1}
f_iv_i=0\), then \(f_1=\ldots=f_{n+1}=0\) \(\meas\)-a.e.\ on \(K\),
which shows that \(v_1,\ldots,v_{n+1}\) are linearly independent on \(K\).
This leads to a contradiction, proving \eqref{eq:equiv_dim}.

In order to prove sufficiency, suppose \eqref{eq:equiv_dim} holds. Then we can find elements
\(v_1,\ldots,v_n\in\mathscr H\) such that for any \(k\leq n\) it holds that \(v_1,\ldots,v_k\)
is a local basis for \(\mathscr H\) on \(E\cap D_k(\mathscr H)\). Therefore, since the set
\(\bigcup_{k\leq n}D_k(\mathscr H)\) covers \(\meas\)-a.a.\ of \(E\) by assumption,
choosing \(\mathcal S\coloneqq\{v_1,\ldots,v_n\}\) it is clear that the local dimension of
\(\mathscr H\) on \(E\) does not exceed \(n\), as required.
\end{proof}

\begin{remark}
It is not clear to us whether a dimensional decomposition
can be built for all \(L^0(\mu)\)-normed \(L^0(\mu)\)-modules,
when \(\mu\) is as in \eqref{eq:def_good_outer_meas} but is not
a Borel measure. One difficulty is, for instance, the fact that
we would need to take \(\mu\)-essential unions; it is not clear
how to do it (in the case where \(\mu\) is the \(2\)-capacity,
this is pointed out in \cite{DebinGigliPasqualetto21}). However, this is not really
relevant for our purposes, since on \(L^0({\rm Cap})\)-normed
\(L^0({\rm Cap})\)-modules we are just interested to upper local
dimension bounds, thus we will not investigate further in this
direction.
\end{remark}

Let \((X,\dist,\meas)\) be a metric measure space and
\(\mu\) be as in \eqref{eq:def_good_outer_meas}. Then we say that
\(\meas\) is \emph{absolutely continuous} with respect to \(\mu\),
shortly \(\meas\ll\mu\), provided it holds that \(\meas(N)=0\)
for every \(N\in\mathscr B(X)\) with \(\mu(N)=0\).
When \(\mu\) is a Borel measure, this notion coincides with
the usual absolute continuity.\\ 
Assuming that \(\meas\ll\mu\), we have
a natural projection map \(\pi_\meas\colon L^0(\mu)\to L^0(\meas)\),
which can be characterized as follows:
\[
\pi_\meas([f]_\mu)=[f]_\meas\quad\text{ for every }
f\colon X\to\setR\text{ Borel.}
\]
One can readily check that the operator \(\pi_\meas\) is
well-defined, linear and continuous. It also induces projection
maps at the level of normed modules, as we are going to describe.\\
Given an \(L^0(\mu)\)-normed \(L^0(\mu)\)-module \(\mathscr M\),
we introduce the following equivalence relation on \(\mathscr M\):
given any \(v,w\in\mathscr M\), we declare that \(v\sim_\meas w\) if
and only if \(\pi_\meas(|v-w|)=0\) holds \(\meas\)-a.e.\ on \(X\).
The resulting quotient space
\(\mathscr M_\meas\coloneqq\mathscr M/\sim_\meas\)
is an \(L^0(\meas)\)-normed \(L^0(\meas)\)-module with respect
to the natural pointwise operations. We denote again by
\(\pi_\meas\colon\mathscr M\to\mathscr M_\meas\) the canonical
projection map, which is linear and continuous.\\ 
Observe that
\[
\mathscr H\;\text{ Hilbert }L^0(\mu)\text{-normed }
L^0(\mu)\text{-module}\;\;\Longrightarrow\;\;
\mathscr H_\meas\;\text{ Hilbert }L^0(\meas)\text{-normed }
L^0(\meas)\text{-module}.
\]
This can be readily checked by just applying \(\pi_\meas\)
to the pointwise parallelogram rule for \(\mathscr H\).
\medskip

As one might expect, an upper local dimension bound passes to
the quotient:
\begin{lemma}\label{lem:local_dim_drops}
Let \((X,\dist,\meas)\) be a metric measure space and let \(\mu\) be as in \eqref{eq:def_good_outer_meas}. Suppose \(\meas\ll\mu\).
Let \(\mathscr H\) be a Hilbert \(L^0(\mu)\)-normed \(L^0(\mu)\)-module
whose local dimension on \(E\in\mathscr B(X)\) does not exceed \(n\in\setN\).
Then the local dimension of \(\mathscr H_\meas\) on \(E\) does not exceed \(n\).
\end{lemma}
\begin{proof}
Let \(\mathcal S\) be a generating subset of \(\mathscr H\) on \(E\) realizing
its upper local dimension bound and consider \(\mathcal S_\meas
\coloneqq\pi_\meas(\mathcal S)\). Fix any \(B\in\mathscr B(X)\) with \(B\subset E\) and
\(\meas(B)>0\), and fix any \(v_1,\ldots,v_{n+1}\in\mathcal S_\meas\). Pick
elements \(w_1,\ldots,w_{n+1}\in\mathcal S\) such that
\(v_i=\pi_\meas(w_i)\). Also, fix any Borel \(\mu\)-a.e.\ representative
\(\bar A_{ij}\colon X\to\setR\) of \(\langle w_i,w_j\rangle\),
thus in particular \(\bar A_{ij}\) is an \(\meas\)-a.e.\ representative
of \(\langle v_i,v_j\rangle=\pi_\meas(\langle w_i,w_j\rangle)\).
We argue by contradiction: suppose \(v_1,\ldots,v_{n+1}\) are linearly
independent on \(B\). Then \autoref{lemma:matrixlin} yields the
existence of a Borel set \(F\subset B\) with \(\meas(F)>0\) (and thus
\(\mu(F)>0\)) such that
\[
{\rm det}\big(\bar A_{ij}(x)\big)_{i,j}\neq 0\quad
\text{ for every }x\in F\, .
\]
Again by \autoref{lemma:matrixlin}, this implies that
\(w_1,\ldots,w_{n+1}\) are linearly independent on \(F\),
thus leading to a contradiction. Therefore, the local dimension
of \(\mathscr H_\meas\) on \(E\) does not exceed \(n\), as required.
\end{proof}

\subsection{\texorpdfstring{$\RCD(K,N)$}{RCD} metric measure spaces and second order calculus}

We assume the reader to be familiar with the language of \(\RCD(K,N)\) spaces
and the notion of pointed measured Gromov--Hausdorff convergence (often
abbreviated to pmGH).

We recall that any sequence \((X_n,\dist_n,\meas_n,x_n)\), \(n\in\setN\) of pointed \(\RCD(K,N)\) spaces such that $\meas_n(B_1(x_n))$ are uniformly bounded and uniformly bounded away from $0$ converges, up to the extraction of a subsequence, to some pointed \(\RCD(K,N)\) space \((X,\dist,\meas,x)\) with respect to the pmGH-topology.	This follows from a compactness argument due to Gromov and the stability of the \(\RCD(K,N)\) condition.

We will repeatedly rely on the convergence and stability properties of function spaces and functions along sequences of $\RCD(K,N)$ metric measure spaces converging in the pointed measured Gromov-Hausdorff sense. We refer to \cite{GigliMondinoSavare15,AmbrosioHonda17} for the basic background about this subject.

\medskip

Gigli in \cite{Gigli18} has developed a second order calculus for $\RCD(K,\infty)$ metric measure spaces, we briefly review the main concepts to fix the notation.

Following \cite{Savare14} we define the space of \emph{test functions} on an \(\RCD(K,\infty)\) space
\((X,\dist,\meas)\) as
\[
{\rm Test}(X)\coloneqq\big\{f\in\mathcal D(\Delta)\cap L^\infty(X)\,:\,|\nabla f|\in L^\infty(X),\,\Delta f\in H^{1,2}(X)\big\} \, .
\]
Test functions are dense in \(H^{1,2}(X)\) and it holds \(\langle\nabla f,\nabla g\rangle\in H^{1,2}(X)\) whenever
\(f,g\in{\rm Test}(X)\).

The \emph{tangent module} \(L^0(TX)\) and the \emph{gradient} \(\bar\nabla\colon{\rm Test}(X)\to L^0(TX)\)
are characterized as follows. The space \(L^0(TX)\) is the Hilbert \(L^0(\meas)\)-normed \(L^0(\meas)\)-module generated by the image of \(\bar\nabla\), while \(\bar\nabla\) is the unique linear operator such that \(|\bar\nabla f|\) coincides \(\meas\)-a.e.\ with \(|\nabla f|\) for every \(f\in{\rm Test}(X)\).
The tangent module $L^2(TX)\subset L^0(TX)$ is the subset of those $v\in L^0(TX)$ such that $|v|\in L^2(\meas)$. We denote by $L^2(T^*X)$ the {\it cotangent module} over $(X,\dist,\meas)$, which is the dual module of $L^2(TX)$.

We recall the notion of \emph{Hessian} of a test function \cite{Gigli18}:
given \(f\in{\rm Test}(X)\), we denote by \({\rm Hess}(f)\) the unique element
of the tensor product \(L^2(T^*X)\otimes L^2(T^*X)\)
(cf.\ \cite[Section 1.5]{Gigli18}) such that
\begin{equation}\label{eq:hess}
\begin{split}
	&\,2\int h\,{\rm Hess}(f)(\nabla g_1\otimes\nabla g_2)\,\di \meas\\
	=&-\int\nabla f\cdot\nabla g_1\,{\rm div}(h\nabla g_2)
	+\nabla f\cdot\nabla g_2\,{\rm div}(h\nabla g_1)
	+h\nabla f\cdot\nabla(\nabla g_1\cdot\nabla g_2)\,\di \meas
\end{split}
\end{equation}
holds for every \(h,g_1,g_2\in{\rm Test}(X)\). The pointwise norm
\(\big|{\rm Hess}(f)\big|\) of \({\rm Hess}(f)\) belongs to \(L^2(\meas)\).

The space $H^{2,2}(X,\dist,\meas)\subset H^{1,2}(X,\dist,\meas)$ is defined by taking the closure of $\Test(X)$ with respect to the norm
\[
	\norm{f}_{H^{2,2}}^2 
	= \norm{f}_{H^{1,2}}^2 + \norm{|\Hess(f)|}_{L^2}^2 \, ,
\]
see \cite[Definition 3.3.1, Definition 3.3.17]{Gigli18}. Let us recall that, as proved in \cite[Proposition 3.3.18]{Gigli18}, we have the inclusion
\begin{equation}\label{eq:dominioLaplacianoH22}
	\mathcal D(\Delta)\subset H^{2,2}(X,\dist,\meas)\, .
\end{equation}

To define $H^{1,2}_C(TX)\subset L^2(TX)$, the space of vector fields with covariant derivative in $L^2$, we follow a similar path. We first introduce the class of test vector fields
\[
	\TestV(X) \coloneqq\left\lbrace \sum_{k=1}^m g_k\nabla f_k \, : \, f_k,g_k\in \Test(X)\, ,\, \, m\in \setN \right\rbrace \, ,
\]
and, employing an identity analogous to \eqref{eq:hess}, we define the covariant derivative $\nabla v$ for any $v\in \TestV(X)$. The Sobolev space $H^{1,2}_C(TX)$ is obtained by taking the closure of $\TestV(X)\subset L^2(TX)$ with respect to the norm
\begin{equation}\label{eq:H12}
	\norm{v}_{H^{1,2}_C(TX)}^2:= \norm{v}_{L^2(TX)}^2 + \norm{|\nabla v|}_{L^2}^2 \, ,
\end{equation}
see \cite[Definition 3.4.1, Definition 3.4.3]{Gigli18}.
\medskip

A powerful tool in the study of non smooth spaces with lower Ricci curvature bounds are the so-called (harmonic) $\delta$-splitting maps. Their use goes back to the seminal works of Cheeger-Colding in the Nineties and more recently they have been employed by the authors in the study of sets of finite perimeter on $\RCD$ spaces in \cite{BruePasqualettoSemola19}. We refer to \cite{BrueNaberSemola20} and \cite{CheegerJiangNaber21} for the sharpest statements available up to now for splitting maps on $\RCD$ and Ricci limit spaces, respectively, even though for the sake of this note the results of \cite{BruePasqualettoSemola19} will be sufficient.

\begin{definition}[Splitting map]\label{def:splitting maps}
	Let \((X,\dist,\meas)\) be an \(\RCD(-1,N)\) space. Let \(x\in X\) and \(\delta>0\)
	be given. Then a map \(u=(u_1,\ldots,u_k)\colon B_r(x)\to\setR^k\) is said to be
	a \emph{\(\delta\)-splitting map} provided:
	\begin{itemize}
		\item[\(\rm i)\)] \(u_a\colon B_r(x)\to\setR\) is harmonic and \(C(N)\)-Lipschitz
		for every \(a=1,\ldots,k\),
		\item[\(\rm ii)\)] \(r^2\fint_{B_r(x)}\big|{\rm Hess}(u_a)\big|^2\,\di \meas\leq\delta\)
		for every \(a=1,\ldots,k\),
		\item[\(\rm iii)\)] \(\fint_{B_r(x)}|\nabla u_a\cdot\nabla u_b
		-\delta_{ab}|\,\di\meas\leq\delta\) for every \(a,b=1,\ldots,k\).
	\end{itemize}
\end{definition}

The following characterization of the quasi continuous representative of Sobolev functions will be of relevance for our purposes, see \cite{KinnunenMartio02,KinnunenLatvala02} dealing with the more general setting of PI spaces.\\
Below we shall denote by $\mathrm{Cap}$ the $2$-capacity of the metric measure space $(X,\dist,\meas)$.

\begin{theorem}\label{thm:caprepr}
	Let $(X,\dist,\meas)$ be an $\RCD(K,N)$ metric measure space for some $K\in\setR$ and $1\le N<\infty$. Let $x\in X$ and $r>0$ be fixed and let $u\in H^{1,2}(B_r(x))$. Then the limit
	\begin{equation}\label{eq:good}
		\lim_{s\to 0}\fint_{B_s(y)}u(z)\di\meas(z)
	\end{equation}
	exists for $\mathrm{Cap}$-a.e.\ $y\in B_r(x)$. Moreover,
	\[
		\lim_{s\to 0}\fint_{B_s(y)}\abs{u(y)-u(z)}^2\di\meas(z)=0\, ,
	\]
	for $\mathrm{Cap}$-a.e.\ $y\in B_r(x)$, where we are considering the $\mathrm{Cap}$-a.e.\ well defined representative of $u$ given by \eqref{eq:good}.
\end{theorem}

Using that $|\nabla u|^2\in H^{1,2}(B_r(x))$ whenever $u:B_{2r}(x) \to \setR$ is harmonic we deduce the following.

\begin{corollary}\label{lemma:gradharm}
	Let $(X,\dist,\meas)$ be an $\RCD(K,N)$ metric measure space for some $K\in\setR$ and $1\le N<\infty$. Let $B_R(x)\subset X$ for some $R>0$ and $x\in X$ and let $u:B_R(x)\to\setR$ be harmonic. Then the limit
	\begin{equation}\label{eq:pointgrad}
		\abs{\nabla u(y)}^2:=\lim_{r\to 0}\fint_{B_r(y)}\abs{\nabla u(z)}^2\di\meas(z) \, ,
	\end{equation} 
	exists for $\mathrm{Cap}$-a.e.\ $y\in B_R(x)$. Moreover
	\[
		\lim_{r\to 0}\fint_{B_r(y)}\abs{\abs{\nabla u(y)}-\abs{\nabla u(z)}}^2\di\meas(z)=0\, .
	\]
\end{corollary}

\begin{remark}
	Relying on the Bochner inequality one can sharpen the conclusion \eqref{eq:pointgrad} by showing the existence of the limit at any point $y\in B_R(x)$. Let us briefly sketch the argument for the reader's convenience.
	
	By Bochner's inequality, $\abs{\nabla u}^2$ has measure valued Laplacian bounded from below by a constant on $B_{R/2}(x)$. Indeed 
	\[
		\boldsymbol{\Delta}\frac{1}{2}\abs{\nabla u}^2\ge -K\abs{\nabla u}^2\meas\, ,\quad\text{on $B_{R/2}(x)$}\, ,
	\]
	see \cite{Savare14,Gigli18}. Moreover, by \cite{Jiang12}, $\abs{\nabla u}$ is bounded on $B_{R/2}(x)$.
	
	Then we can recall that, on general metric measure spaces supporting doubling and Poincar\'e inequalities, subharmonic functions have Lebesgue points everywhere,
	see \cite[Proposition 8.24]{BjornBjorn11}. In our setting, this regularity result extends to functions with Laplacian locally bounded from below, since we can always perturb them to locally subharmonic functions by adding a solution of $\Delta v= c$ on $B_{r}(x)$ for any given $c\in\setR$ and, again by \cite{Jiang12}, $v$ is locally Lipschitz.
\end{remark}

Besides quasi-continuous functions, quasi-continuous vector fields play an important role in the theory of \(\RCD(K,N)\) spaces.
This is made precise by the notion of \emph{tangent \(L^0({\rm Cap})\)-module} \(L^0_{\rm Cap}(TX)\) introduced in \cite[Theorem 2.6]{DebinGigliPasqualetto21},
which we are going to recall. First of all, let us stress again that \(\langle\nabla f,\nabla g\rangle\in H^{1,2}(X)\) whenever
\(f,g\in{\rm Test}(X)\). In particular, the function \(\langle\nabla f,\nabla g\rangle\) (and thus also \(|\nabla f|^2\)) admits a quasi-continuous representative.
\medskip

The \emph{capacitary tangent module} \(L^0_{\rm Cap}(TX)\) and the \emph{capacitary gradient} \(\bar\nabla\colon{\rm Test}(X)\to L^0_{\rm Cap}(TX)\)
can be characterized as follows. The space \(L^0_{\rm Cap}(TX)\) is the Hilbert \(L^0({\rm Cap})\)-normed \(L^0({\rm Cap})\)-module generated by
the image of \(\bar\nabla\), while \(\bar\nabla\) is a linear operator such that \(|\bar\nabla f|\) coincides \({\rm Cap}\)-a.e.\ with any quasi-continuous
representative of \(|\nabla f|\) for every \(f\in{\rm Test}(X)\).

\subsection{Structure theory for \texorpdfstring{$\RCD(K,N)$}{RCD} spaces}

Let us briefly review the main results concerning the state of the art about the so-called structure theory of $\RCD(K,N)$ spaces.

Given a m.m.s.\ $(X,\dist,\meas)$, $x\in X$ and $r\in(0,1)$, we consider the rescaled and normalized pointed m.m.s.\ $(X,r^{-1}\dist,\meas_r^{x},x)$, where 
\[
\meas_r^x:= \left( \int_{B_r(x)} \left(1-\frac{\dist(x,y)}{r}\right) \di \meas(y)\right)^{-1}\meas=C(x,r)^{-1}\meas\, .
\]

\begin{definition}
	We say that a pointed m.m.s.\ $(Y,\dist_Y,\eta,y)$ is tangent to $(X,\dist,\meas)$ at $x$ if there exists a sequence $r_i\downarrow 0$ such that $(X,r_i^{-1}\dist,\meas_{r_i}^x,x)\rightarrow(Y,\dist_Y,\eta,y)$ in the pmGH-topology. The collection of all the tangent spaces of $(X,\dist,\meas)$ at $x$ is denoted by $\Tan_x(X,\dist,\meas)$.
\end{definition}

A compactness argument, which is originally due to Gromov, together with the rescaling and stability properties of the $\RCD(K,N)$ condition, yields that $\Tan_x(X,\dist,\meas)$ is non-empty for every $x\in X$ and its elements are all $\RCD(0,N)$ pointed m.m.\ spaces.\\
Let us recall below the notion of $k$-regular point and $k$-regular set. 

\begin{definition}\label{def:regular point}
	Given any natural $1\le k\le N$, we say that $x\in X$ is a $k$-regular point if
	\[
		\Tan_x(X,\dist,\meas)=\left\lbrace (\setR^k,\dist_{eucl},c_k\Leb^k,0)  \right\rbrace\, .
	\]
	We shall denote by $\mathcal{R}_k$ the set of $k$-regular points in $X$.
\end{definition}

\begin{theorem}\label{thm:structure theory}
	Let $(X,\dist,\meas)$ be an $\RCD(K,N)$ m.m.s.\ with $K\in\setR$ and $1\le N<\infty$. Then there exists a natural number $1\le n\le N$, called essential dimension of $X$, such that $\meas(X\setminus \mathcal{R}_n)=0$. Moreover $\mathcal{R}_n$ is $(\meas,n)$-rectifiable and $\meas$ is representable as $\theta\haus^n\res {\mathcal{R}_n}$ for some nonnegative density $\theta\in L^1_{\loc}(X,\haus^n\res\mathcal{R}_n)$.
\end{theorem}

The rectifiability above was obtained in \cite{MondinoNaber19}, while the behaviour of the reference measure was studied in the independent works \cite{KellMondino18,DePhilippisMarcheseRindler17,GigliPasqualetto16a}.
The constancy of the dimension was obtained in \cite{BrueSemola20} partially generalizing the previous \cite{ColdingNaber12}, which was dealing with Ricci limit spaces. 
In the more recent \cite{Deng20} the results of \cite{ColdingNaber12} have been fully generalized to the $\RCD$ framework. In particular we will rely on the following consequence of the main result of \cite{Deng20}.

\begin{proposition}\label{prop:a.e.full}
Let $(X,\dist,\meas)$ be an $\RCD(K,N)$ metric measure space whose essential dimension is $1\le n\le N$. Let $\gamma:[0,1]\to X$ be a geodesic such that $\gamma(t)\in \mathcal{R}_n$ for a.e.\ $t\in (0,1)$. Then $\gamma(t)\in \mathcal{R}_n$ for every $t\in (0,1)$.
\end{proposition}

\begin{proof}
By \cite{Deng20} tangent cones coming from the same sequence of scaling radii $r_i\downarrow 0$ are continuous w.r.t.\ the pGH topology in the interior of minimizing geodesics. Since $\gamma(t)\in\mathcal{R}_n$ for a.e.\ $t\in(0,1)$ it follows that any tangent cone at any $\gamma(s)$ for $s\in (0,1)$ is isometric as a metric space to $(\mathbb{R}^n, \dist_{eucl})$. 
An iterative application of the splitting theorem \cite{Gigli13} yields that any $\RCD(0,N)$ metric measure space isometric to $(\mathbb{R}^n, \dist_{eucl})$ is actually isomorphic as a metric measure space to $(\setR^n, \dist_{eucl}, c\Leb^n)$. This implies that $\gamma(s)\in\mathcal{R}_n$ for any $s\in (0,1)$.
\end{proof}

\subsection{Sets of finite perimeter}
Here we recall the basic background about sets of finite perimeter on metric measure spaces. Then we present some more recent results obtained in the setting of $\RCD(K,N)$ spaces.

\begin{definition}[Function of bounded variation]\label{def:bvfunction}
	We say that $f\in L^1(X,\meas)$ belongs to the space $\BV(X,\dist,\meas)$ of functions of bounded variation if there exist
	locally Lipschitz functions $f_i$ converging to $f$ in $L^1(X,\meas)$ such that
	\begin{equation*}
	\limsup_{i\to\infty}\int_X\abs{\nabla f_i}\di \meas<+\infty\, .
	\end{equation*}
	\end{definition}  
	If $f\in \BV(X,\dist,\meas)$ one can define 
	\begin{equation*}
	\abs{Df}(A):= \inf\left\lbrace \liminf_{i\to\infty}\int_A\abs{\nabla f_i}\di \meas: f_i\in\Lip_{\loc}(A),\quad f_i\to f \text{ in } L^1(A,\meas)\right\rbrace\, ,  
	\end{equation*}
	for any open $A\subset X$. This set function 
	is the restriction to open sets of a finite Borel measure that we call \emph{total variation of $f$} and still denote by $\abs{Df}$.

Dropping the global integrability condition on $f=\chi_E$, let us recall now the analogous definition of set of finite perimeter 
in a metric measure space (see again \cite{Am02,Miranda03,AmbrosioDiMarino14}).

\begin{definition}[Perimeter and sets of finite perimeter]
	\label{def:setoffiniteperimeter}
	Given a Borel set $E\subset X$ and an open set $A$ the perimeter $\Per_E(A)$ is defined in the following way:
	\begin{equation*}
	\Per_E(A):=\inf\left\lbrace \liminf_{n\to\infty}\int_A\abs{\nabla u_n}\di\meas\, :\,  u_n\in\Liploc(A),\quad u_n\to\chi_E\quad \text{in } L^1_{{\loc}}(A,\meas)\right\rbrace\,  .
	\end{equation*}
	We say that $E$ has finite perimeter if $\Per_E(X)\, <\, \infty$. In that case it can be proved that the set function $A\mapsto\Per_E(A)$ is the restriction to open sets of a finite Borel measure $\Per_E$ defined by
	\begin{equation*}
	\Per_E(B):=\inf\left\lbrace \Per_E(A): B\subset A,\text{ } A \text{ open}\right\rbrace\, .
	\end{equation*}
\end{definition}

Let us remark for the sake of clarity that $E\subset X$ with finite $\meas$-measure is a set of finite perimeter if and only if $\chi_E\in \BV(X,\dist,\meas)$ and that $\Per_E=\abs{D\chi_E}$. In the following we will say that $E\subset X$ is a set of locally finite perimeter if $\chi_E$ is a function of locally bounded variation, that is to say $\eta\chi_E\in \BV(X,\dist,\meas)$ for any $\eta\in \Lipbs(X,\dist)$.
\medskip

We recall that a coarea formula holds in this generality, see
\cite[Proposition 4.2]{Miranda03}, dealing with locally compact spaces, and its proof works in the more general setting of metric measure spaces.

\begin{theorem}\label{thm:coarea}
Let $v\in\BV(X,\dist,\meas)$.
Then, $\{v>r\}$ has finite perimeter for $\Leb^1$-a.e.\ $r\in\setR$ and, for any Borel function $f:X\to[0,+\infty]$, it holds 
\begin{equation}\label{eq:coarea}
\int_X f\di\abs{Dv}=\int_{-\infty}^{+\infty}\left(\int_X f\di\Per_{\{v>r\}}\right)\di r\,  .
\end{equation}
\end{theorem}

In \cite{Am02} a general theory of sets of finite perimeter on metric measure spaces satisfying doubling and Poincar\'e inequalities (the so-called PI spaces) was developed. The following asymptotic doubling property of the perimeter measure will be of some relevance for our purposes.

\begin{proposition}[Corollary 5.8 in \cite{Am02}]\label{prop:asymdoub}
Let $(X,\dist,\meas)$ be a PI metric measure space and let $E\subset X$ be a set of finite perimeter. Then $\abs{D\chi_E}$ is asymptotically doubling, i.e.\
\[
\limsup_{r\to 0}\frac{\abs{D\chi_E}(B_{2r}(x))}{\abs{D\chi_E}(B_{r}(x))}<\infty\, ,\quad\text{for $\abs{D\chi_E}$-a.e.\ $x\in X$}\, .
\]
\end{proposition} 

An important consequence of the asymptotically doubling property is the validity of the Lebesgue differentiation theorem, i.e.\ for any $f\in L^1(|D\chi_E|)$, for $|D\chi_E|$-a.e.\ $x\in X$ it holds
\[
	\lim_{r\to 0}  \fint_{B_r(x)} |f(y) - f(x)| \di |D\chi_E|(y) = 0 \, .
\]
It can be proven by observing that a countable sub-family of $\{G_{r,M} \, : \, r>0,\, M>1\}$ covers $X$ up to a $|D\chi_E|$-negligible set, where
\[
	G_{r,M} := \{ x\in X : |D\chi_E|(B_{2s}(x)) \le M |D\chi_E|(B_s(x)) \, \, \text{for any $0<s<r$} \} \, .
\]
On each $G_{r,M}$ we can apply the standard Lebesgue differentiation theorem for doubling metric measure spaces.

Let us recall the notion of tangent to a set of finite perimeter that has been introduced in \cite{AmbrosioBrueSemola19}.

\begin{definition}\label{def:blowupset}
	Let $(X,\dist,\meas)$ be an $\RCD(K,N)$ m.m.s., $x\in X$ and let $E\subset X$ be a set of locally finite perimeter. 
	We denote by $\Tan_x(X,\dist,\meas,E)$ the collection of quintuples $(Y,\varrho,\mu,y, F)$ satisfying the following two properties:
	\begin{itemize}
		\item[(a)] $(Y,\varrho,\mu,y)\in\Tan_x(X,\dist,\meas)$ and $r_i\downarrow 0$ are such that the rescaled spaces $(X,r_i^{-1}\dist,\meas_x^{r_i},x)$ converge to $(Y,\varrho,\mu,y)$ in the pointed measured Gromov-Hausdorff topology;
		\item[(b)] $F$ is a set of locally finite perimeter in $Y$ with $\mu(F)>0$ and, if $r_i$ are as in $\rm(a)$, then the sequence $f_i=\chi_E$ converges in $L^1_{\loc}$ to $\chi_F$.
	\end{itemize}
\end{definition}

In \cite{AmbrosioBrueSemola19,BruePasqualettoSemola19} the following partial generalization of De Giorgi's classical regularity theorem for boundaries of sets of finite perimeter has been obtained in the setting of $\RCD(K,N)$ metric measure spaces.

\begin{theorem}[Theorem 3.2 and Theorem 4.1 in \cite{BruePasqualettoSemola19}]\label{thm:DGRCD}
	Let $(X,\dist,\meas)$ be an $\RCD(K,N)$ m.m.s.\ with essential dimension $1\le n\le N$, $E\subset X$ be a set of finite perimeter. Then, for $|D\chi_E|$-a.e.\ $x\in X$, there exists $k=1,\ldots,n$ such that
	\[
	\Tan_x(X,\dist,\meas,E)=\left\lbrace (\setR^k,\dist_{eucl},c_k\Leb^k,0^k,\left\lbrace x_k>0\right\rbrace )\right\rbrace\, .
	\]
Moreover, setting 
\[
\mathcal{F}_kE:= \left\lbrace x\in X\, :\, \Tan_x(X,\dist,\meas,E)= \left\lbrace (\setR^k,\dist_{eucl},c_k\Leb^k,0^k,\left\lbrace x_k>0\right\rbrace )\right\rbrace \right \rbrace\, ,
\]
it holds that $\mathcal{F}_kE$ is $(\abs{D\chi_E}, k)$-rectifiable for any $k=1,\dots, n$.
\end{theorem}

In \cite{BruePasqualettoSemola19} a notion of \emph{restriction} of the tangent module $L^2(TX)$ over the boundary of a set of finite perimeter has been introduced. We recall below the relevant terminology.
Let \(E\subset X\) be a given set of finite perimeter. As we proved in \cite{BruePasqualettoSemola19}, it holds
that \(|D\chi_E|\ll{\rm Cap}\), thus in particular we have a natural trace operator \({\rm tr}_E\colon H^{1,2}(X)\to L^0(|D\chi_E|)\)
over the boundary of \(E\), obtained by sending each function \(f\in H^{1,2}(X)\) to the equivalence class (up to
\(|D\chi_E|\)-a.e.\ equality) of any quasi-continuous representative of \(f\). Concerning vector fields, we have the following result.

\begin{theorem}[Tangent module over \(\partial E\), Theorem 2.2 in \cite{BruePasqualettoSemola19}]\label{thm:tg_mod_over_bdry}
Let \((X,\dist,\meas)\) be an \(\RCD(K,N)\) space. Let \(E\subset X\)
be a set of finite perimeter.
Then there exists a unique couple \(\big(L^2_E(TX),\bar\nabla\big)\) -- where
\(L^2_E(TX)\) is an \(L^2(|D\chi_E|)\)-normed \(L^\infty(|D\chi_E|)\)-module
and \(\bar\nabla:\,{\rm Test}(X)\to L^2_E(TX)\) is linear -- such that:
\begin{itemize}
\item[\(\rm i)\)] The equality \(|\bar\nabla f|=\tr_E(|\nabla f|)\) holds
\(|D\chi_E|\)-a.e.\ for every \(f\in{\rm Test}(X)\).
\item[\(\rm ii)\)] \(\big\{\sum_{i=1}^n\chi_{E_i}\bar\nabla f_i\;
\big|\;(E_i)_{i=1}^n\text{ Borel partition of }X,
\,(f_i)_{i=1}^n\subset{\rm Test}(X)\big\}\) is dense in \(L^2_E(TX)\).
\end{itemize}
The space \(L^2_E(TX)\) is called
\emph{tangent module over the boundary of \(E\)} and \(\bar\nabla\)
is the \emph{gradient}.
\end{theorem}

The space \(L^2_E(TX)\) was in fact obtained as the family of all (\(2\)-integrable) sections of the quotient
\[
L^0_E(TX)\coloneqq L^0_{\rm Cap}(TX)_{|D\chi_E|},
\]
where we adopt the notation for quotient modules that we introduced before \autoref{lem:local_dim_drops};
this comment will play a role in \autoref{rmk:equiv_nu_E=nu_F}. Observe also that \(L^2_E(TX)\) is a Hilbert module.

\medskip

The notion of restriction of the tangent module over the boundary of a set of finite perimeter is a key tool to prove a Gauss-Green integration by parts formula in this setting, for sufficiently regular vector fields.  

\begin{theorem}[Theorem 2.4 in \cite{BruePasqualettoSemola19}]\label{thm:GGRCDsmooth}
Let $(X,\dist,\meas)$ be an $\RCD(K,N)$ metric measure space and let $E\subset X$ be a set with finite perimeter and finite measure. Then there exists a unique vector field $\nu_E\in L^2_E(TX)$ such that $\abs{\nu_E}=1$ holds $\Per_E$-a.e.\ and
\begin{equation*}
\int _E\div v\di\meas
=-
\int \langle \mathrm{tr}_E v,\nu_E \rangle \di\Per_E\, , \quad\text{for any $v\in H^{1,2}_C(TX)\cap D(\div )$ such that $\abs{v}\in L^{\infty}(\meas)$}\, .
\end{equation*}
\end{theorem}

\begin{remark}
For the sake of clarity, let us remark that for a smooth domain on a smooth Riemannian manifold, the vector field $\nu_E$ above would correspond to the interior unit normal vector field.
\end{remark}
Let us remark that one of the goals of the present note is to lower the regularity assumptions on the vector fields in the Gauss-Green formula above.

\subsection{Geodesic plans and functions with bounded variation}

	In this section we present a characterization of BV functions in terms of {\it test plans} concentrated on geodesics. It will be relevant for the purpose of proving that boundaries of sets with finite perimeter have constant dimension.
\medskip

Let us begin by recalling the concept of test plan,
which was introduced in \cite{AmbrosioGigliSavare13}.
Given a metric measure space \((X,\dist,\meas)\) and a Borel probability
measure \(\ppi\) on the space of continuous curves \(C([0,1],X)\), we say that \(\ppi\) is an \emph{\(\infty\)-test plan}
if it is concentrated on an equi-Lipschitz family of curves and there exists \(C>0\) such that
\[
({\rm e}_t)_*\ppi\leq C\meas\quad\text{ for every }t\in[0,1]\, ,
\]
where \({\rm e}_t\colon C([0,1],X)\to X\) stands for \(\gamma\mapsto{\rm e}_t(\gamma)\coloneqq\gamma_t\).
The minimal such \(C\) is denoted by \({\rm Comp}(\ppi)\), while \(\Lip(\ppi)\) stands for the
minimal \(L\geq 0\) such that \(\ppi\) is concentrated on \(L\)-Lipschitz curves.
\medskip

The following notion of function having bounded \(\Pi\)-variation, where \(\Pi\) is an arbitrary family
of \(\infty\)-test plans, has been recently proposed in \cite[Definition 2.7]{NobiliPasqualettoSchultz21},
as a generalization of the notion of BV space via test plans introduced in \cite[Section 5.3]{AmbrosioDiMarino14}.

\begin{definition}\label{def:BV}
Let \((X,\dist,\meas)\) be a metric measure space. Let \(\Pi\) be a
family of \(\infty\)-test plans on \(X\). Let \(f\in L^1(\meas)\)
be given. Then we declare that \(f\in{\BV}_\Pi(X)\) provided:
\begin{itemize}
\item[\(\rm i)\)] Given any \(\ppi\in\Pi\), it holds that
\(f\circ\gamma\in{\BV}(0,1)\) for \(\ppi\)-a.e.\ \(\gamma\).
\item[\(\rm ii)\)] There exists a finite Borel measure \(\mu\geq 0\)
on \(X\) such that for every \(\ppi\in\Pi\) it holds that
\begin{equation}\label{eq:def_BV}
\int\gamma_\#|D(f\circ\gamma)|(B)\,\di\ppi(\gamma)\leq
{\rm Comp}(\ppi){\Lip}(\ppi)\mu(B),\quad\text{ for every }
B\subset X\text{ Borel.}
\end{equation}
\end{itemize}
The minimal measure \(\mu\geq 0\) satisfying \eqref{eq:def_BV}
is denoted by \(|Df|_\Pi\).
\end{definition}

The following theorem follows from the results in \cite{NobiliPasqualettoSchultz21}.
It says that, on finite-dimensional \(\RCD\) spaces, (a countable family of) \(\infty\)-test plans
concentrated on geodesics is sufficient to recover all BV functions and their total variation
measures (up to a given multiplicative constant).

\begin{theorem}\label{thm:master_tp}
Let \((X,\dist,\meas)\) be an \({\RCD}(K,N)\) space, where \(N<\infty\).
Then there exists a sequence \(\Pi=(\ppi_i)_i\) of \(\infty\)-test plans
on \(X\) concentrated on \({\rm Geo}(X)\) such that
\({\BV}_\Pi(X)={\BV}(X)\) and
\begin{equation}\label{eq:ineq_master_tp}
|Df|_\Pi\leq|Df|\leq 2^N|Df|_\Pi\, ,\quad\text{ for every }f\in\BV(X)\, .
\end{equation}
\end{theorem}
A few words about \autoref{thm:master_tp} are in order. Calling \(\Pi_{\rm Geo}\) the
family of those \(\infty\)-test plans on \(X\) that are concentrated on \({\rm Geo}(X)\), we know
from \cite[Remark 3.8]{NobiliPasqualettoSchultz21} that \(\BV_{\Pi_{\rm Geo}}(X)=\BV(X)\) and
\(|Df|_{\Pi_{\rm Geo}}\leq|Df|\leq 2^N|Df|_{\Pi_{\rm Geo}}\) for every \(f\in\BV(X)\).
Here, the fact that \(X\) is non-branching is ensured by \cite[Theorem 1.3]{Deng20}.
Moreover, by using \cite[Theorem 3.9 and Remark 3.10]{NobiliPasqualettoSchultz21} we can select a countable
subfamily \(\Pi\subset\Pi_{\rm Geo}\) verifying the statement of \autoref{thm:master_tp}.\\
Albeit not strictly needed for our purposes, we point out that in the case where
either \(K=0\) or \(X\) is compact, \eqref{eq:ineq_master_tp} improves to the identity
\(|Df|_\Pi=|Df|\) for every \(f\in\BV(X)\). In fact, it is strongly believed that this is
actually the case on every \(\RCD(K,N)\) space. 
\section{Constant dimension along boundaries}\label{sec:constancy}
In this section we prove that the total variation measure of $BV$ functions is concentrated on the regular set $\mathcal{R}_n$, where $n$ is the essential dimension of the space. 
This result, heavily relying on the continuity of tangent cones along geodesics proven recently in \cite{Deng20}, allows us to positively answer one of the questions left open in \cite{BruePasqualettoSemola19}.
We can indeed show that boundaries of sets of finite perimeter are concentrated on $\mathcal{R}_n$, hence they have constant dimension $n-1$. 

An important consequence is that the tangent module $L^2_E(TX)$ over the boundary of a set with finite perimeter $E$ has dimension $n$ as well.

\begin{theorem}\label{thm:tv_on_reg_set}
Let \((X,\dist,\meas)\) be an \({\RCD}(K,N)\) space for some $K\in\setR$ and \(1\le N<\infty\).
Let \(f\in{\BV}(X)\) be given. Then \(|Df|\) is concentrated
on \(\mathcal R_n\), where \(1\le n\leq N\) stands for the essential dimension
of \(X\).
\end{theorem}
\begin{proof}
Fix a Borel representative \(\bar f\colon X\to\setR\) of \(f\).
Let \((\ppi_i)_i\) be chosen as in \autoref{thm:master_tp}.
Calling \(\e\colon C([0,1],X)\times[0,1]\to X\) the evaluation
map \((\gamma,t)\mapsto\e(\gamma,t)\coloneqq\gamma_t\), for any
\(i\in\setN\) we have that
\(\e_\#(\ppi_i\otimes\mathcal L^1)\leq{\rm Comp}(\ppi_i)\meas\).
Since \(\meas(X\setminus\mathcal R_n)=0\), we deduce that
\[
\gamma_t\in\mathcal R_n,\quad\text{ for }
(\ppi_i\otimes\mathcal L^1)\text{-a.e.\ }(\gamma,t).
\]
Recalling the fact that each \(\ppi_i\) is concentrated on geodesics,
as well as Ulam Lemma,
we can find a \(\sigma\)-compact set
\(\Gamma_i\subset{\rm Geo}(X)\) having the property that for
every \(\gamma\in\Gamma_i\) the following conditions hold:
\[
\bar f\circ\gamma\in{\BV}(0,1),
\quad
\ppi_i(\Gamma_i^c)=0,
\quad
\gamma_t\in\mathcal R_n\;\text{ for }\mathcal L^1\text{-a.e.\ }t\in[0,1].
\]
The last condition implies that \(\gamma_t\in\mathcal R_n\)
for every \(\gamma\in\Gamma_i\) and \(t\in(0,1)\) by \autoref{prop:a.e.full}. Define
\[
G\coloneqq\bigcup_{i\in\setN}G_i,\quad\text{ where we set }
G_i\coloneqq\big\{\gamma_t\;\big|\;\gamma\in\Gamma_i,\,t\in(0,1)\big\}.
\]
Observe that \(G\subset\mathcal R_n\). Moreover,
each set \(G_i\), being the continuous image (under \(e\))
of the \(\sigma\)-compact set \(\Gamma_i\times(0,1)\), is
\(\sigma\)-compact itself and accordingly \(G\) is Borel.
Therefore, it only remains to show that
\(|Df|(X\setminus G)=0\). Since \(\Pi\coloneqq(\ppi_i)_i\)
satisfies \(|Df|\leq 2^N|Df|_\Pi\), we just have to check that
\begin{equation}\label{eq:tv_on_reg_set_aux}
\int\gamma_\#|D(f\circ\gamma)|(X\setminus G)\,\di\ppi_i(\gamma)=0,
\quad\text{ for every }i\in\setN.
\end{equation}
To prove it, notice that \(\gamma^{-1}(X\setminus G_i)\cap(0,1)
=\emptyset\) for every \(\gamma\in\Gamma_i\) by definition of \(G_i\),
so that
\[\begin{split}
\int\gamma_\#|D(f\circ\gamma)|(X\setminus G)\,\di\ppi_i(\gamma)
&=\int_{\Gamma_i}\gamma_\#|D(\bar f\circ\gamma)|(X\setminus G)
\,\di\ppi_i(\gamma)\\
&\leq\int_{\Gamma_i}\gamma_\#|D(\bar f\circ\gamma)|(X\setminus G_i)
\,\di\ppi_i(\gamma)\\
&=\int_{\Gamma_i}|D(\bar f\circ\gamma)|\big(\gamma^{-1}(X\setminus G_i)\big)\,\di\ppi_i(\gamma)=0.
\end{split}\]
This yields \eqref{eq:tv_on_reg_set_aux} and accordingly the statement.
\end{proof}

We shall denote by $\mathcal{H}^h$ the codimension $1$ Hausdorff-type measure on $(X,\dist,\meas)$ build through the Carath\'{e}odory construction with gauge function $h(B_r(x))=\meas(B_r(x))/r$, see \cite[Definition 1.9]{BruePasqualettoSemola19}.
In \cite[Corollary 3.15]{BruePasqualettoSemola19} it was proved that, for a set of locally finite perimeter $E\subset X$,
\begin{equation}\label{eq:reprnoncost}
	\abs{D \chi_E}=\sum_{k=1}^n\frac{\omega_{k-1} }{\omega_k}\mathcal{H}^h\res \mathcal{F}_kE\, ,
\end{equation}
where $1\le n\le N$ denotes the essential dimension of $(X,\dist,\meas)$.

\begin{corollary}\label{cor:repr}
Let \((X,\dist,\meas)\) be an \({\RCD}(K,N)\) space with essential dimension $n\le N$. Let $E\subset X$ be a set of finite perimeter.
Then $\Per_E$ is concentrated on $\mathcal{R}_n$. In particular $\mathcal{F}_k E$ is $|D\chi_E|$-negligible for $k\neq n$ and 
	\begin{equation}\label{eq:reprcdim}
		\abs{D\chi_E}=\frac{\omega_{n-1}}{\omega_n}\mathcal{H}^h\res \mathcal{F}_nE\,.
	\end{equation}
\end{corollary}

When $(X,\dist, \haus^N)$ is a non collapsed $\RCD(K,N)$ space one has $\haus^h = \frac{\omega_N}{\omega_{N-1}} \haus^{N-1}$ and the identity \eqref{eq:reprcdim} was already proven in \cite{AmbrosioBrueSemola19,BruePasqualettoSemola19}.

Let us point out a remarkable consequence of \autoref{thm:tv_on_reg_set} about noncollapsing of codimension one hypersurfaces relative to the ambient spaces, even when the ambient manifolds do collapse.

\begin{corollary}\label{cor:noncoll}
Let $K\in\setR$ and $N\ge 2$ be fixed. Let us consider a sequence of pointed smooth $N$-dimensional Riemannian manifolds $(M_i,\dist_i,\haus^N,p_i)$ with Ricci curvature uniformly bounded from below by $K$ and assume that $(M_i,\dist_i,\haus^N/\haus^N(B_1(p_i)),p_i)$ converge in the pmGH topology to a Ricci limit $(X,\dist,\meas,p)$ with essential dimension $1\le n\le N$. Let moreover $\Omega_i\subset M_i$ be open domains with smooth boundary such that 
\begin{equation}\label{eq:volbd}
\frac{1}{4}\le \frac{\haus^N(\Omega_i\cap B_2(p_i))}{\haus^N(B_2(p_i))}\le \frac{3}{4}\, ,\quad\text{for any $i\in\setN$}
\end{equation}
and 
\begin{equation}\label{eq:areabd}
\frac{\haus^{N-1}(\partial\Omega_i\cap B_2(p_i))}{\haus^N(B_2(p_i))}\le C\, ,\quad\text{for any $i\in\setN$, for some $C>0$}\, .
\end{equation}
Then the perimeter measure $\abs{D\chi_{\Omega}}$ of any limit $\Omega\subset X$ of the sequence $\Omega_i$ in the $L^1_{\loc}$ sense is concentrated on the $n$-regular set $\mathcal{R}_n$ of $X$ on $B_1(p)$.
\end{corollary}

\begin{proof}
Let us observe that the sequence $\Omega_i$ admits limit points with respect to $L^1_{\loc}$ convergence, thanks to \eqref{eq:volbd}, \eqref{eq:areabd} and \cite[Corollary 3.4]{AmbrosioBrueSemola19}. The conclusion then follows from \autoref{thm:tv_on_reg_set}.
\end{proof}

\subsection{Dimension of the tangent module over the boundary of \texorpdfstring{$E$}{E}}

Recall that, by \cite{GigliPasqualetto16}, almost everywhere constancy of the dimension at the level of pointed measured GH tangent spaces on $\RCD(K,N)$ spaces can be turned into constancy of the dimension for the tangent module $L^2(TX)$. 

Here we wish to show that the same phenomenon occurs, for slightly different reasons, also at the level of the restriction of the tangent module to the boundary of any set of finite perimeter, as introduced in \cite{BruePasqualettoSemola19}, see \autoref{thm:tg_mod_over_bdry} for the relevant definition.

\begin{theorem}\label{thm:constdim tanmod}
	Let $(X,\dist,\meas)$ be an $\RCD(K,N)$ metric measure space for some $K\in\setR$ and $1\le N<\infty$. Let $E\subset X$ be a set of finite perimeter and let $1\le n\le N$ be the essential dimension of $(X,\dist,\meas)$. Then the dimension of $L^2_E(TX)$ is constant and equals $n$.
\end{theorem}

The proof of \autoref{thm:constdim tanmod} is given in two steps. We first show that the dimension of $L^2_E(TX)$ is smaller than $n$.

\begin{proposition}\label{prop:indtest}
	Let $(X,\dist,\meas)$ be an $\RCD(K,N)$ metric measure space with essential dimension $1\le n\le N$. Let $E\subset X$ be a set of locally finite perimeter.
     Then the local dimension of \(L^0_{\rm Cap}(TX)\) on \(\mathcal F_n E\) does not exceed \(n\).
    In particular, the local dimension of \(L^2_E(TX)\) is smaller than or equal to \(n\).
\end{proposition}

To prove that the dimension of $L^2_E(TX)$ is bigger than or equal to $n$ we employ the first conclusion of the following proposition and the fact that $|D\chi_E|(\mathcal{F}_k E)=0$ for $k<n$ as a consequence of \autoref{cor:repr}.
Indeed, given $x\in X$, $r_x>0$ and $(u_i):B_{r_x}(x)\to\setR^n$ as in \eqref{eq:ortomper}, it is simple to check that
\[
	A_{i,j} := \nabla u_i \cdot \nabla u_j
\]
is invertible $|D\chi_E|$-a.e.\ in $B_r(x)$ for $r<r_x$ small enough. Hence we can apply \autoref{lemma:matrixlin} and conclude that the dimension of $L^2_E(TX)$ is bigger than $n$.

\begin{proposition}\label{prop:harmcoord}
	Let $(X,\dist,\meas)$ be an $\RCD(K,N)$ metric measure space for some $K\in\setR$ and $1\le N<\infty$ and let $1\le n\le N$ denote its essential dimension. Let $E\subset X$ be a set of finite perimeter. Then, for $\abs{D\chi_E}$-a.e.\ $x\in X$ there exist $r_x>0$ and harmonic functions $(u_i):B_{r_x}(x)\to\setR^n$ such that, for any $i,j\in\{1,\dots, n\}$, 
	\begin{equation}\label{eq:ortomper}
		\lim_{r\to 0}\fint_{B_r(x)}\abs{\nabla u_i\cdot\nabla u_j-\delta_{ij}}\di\meas=\lim_{r\to 0}\fint_{B_r(x)}\abs{\nabla u_i\cdot\nabla u_j-\delta_{ij}}\di\abs{D\chi_E}=0\, .
	\end{equation}
	Moreover, for any $i\in\{1,\dots, n\}$ the following limits exist:
	\begin{equation}\label{eq:exsisttracenu}
		\nu_i:=\lim_{r\to 0}\fint_{B_r(x)}\nu_E\cdot\nabla u_i\di\abs{D\chi_E},
	\end{equation}
	\begin{equation}\label{eq:reinflebper}
		\lim_{r\to 0}\fint_{B_r(x)}\abs{\nu_i-\nu_E\cdot\nabla u_i}\di\abs{D\chi_E}=0
	\end{equation}
	and, setting $\nu:=(\nu_1,\dots,\nu_n)$, it holds $\abs{\nu}_{\setR^n}=1$.
\end{proposition}

The last two conclusions of \autoref{prop:harmcoord} will play a role in \autoref{sec:cutpaste}. To understand its meaning let us recall that for sets of finite perimeter $E\subset \setR^n$ a point in $\partial ^*E$ is a reduced boundary point if the limit
\begin{equation}\label{eq:differen}
	\nu_E(x):=\lim_{r\to 0}\frac{D\chi_E(B_r(x))}{\abs{D\chi_E}(B_r(x))}
\end{equation}
exists and $\abs{\nu_E(x)}=1$, where we recall that $D\chi_E$ is the distributional derivative of $\chi_E$, which is a Radon measure under the assumption that $E$ has finite perimeter.\\ 
In \cite{AmbrosioBrueSemola19,BruePasqualettoSemola19} there was the necessity to argue differently, due to the absence of a Besicovitch differentiation theorem and of a notion of distributional derivative. The properties \eqref{eq:exsisttracenu} and \eqref{eq:reinflebper} are the natural replacement of \eqref{eq:differen} in this framework. 

\medskip

We conclude this section by proving \autoref{prop:indtest} and \autoref{prop:harmcoord}.

\begin{proof}[Proof of \autoref{prop:indtest}]
	The proof is based on a blow-up argument. Assuming linear independence of the vector fields $\nabla f_1,\dots,\nabla f_{n+1}$ we find a point $x\in \mathcal{F}_n E$ where the blow-ups of the functions $f_1,\dots,f_{n+1}$ are linearly independent linear harmonic functions. This will contradict the fact that the blow-up of $(X,\dist,\meas)$ at $x$ is $(\setR^n, \dist_{eucl}, \Leb^n, 0^n)$.
	\medskip
	
   We argue by contradiction.
   Suppose there exist a Borel set $B\subset\mathcal F_n E$ with ${\rm Cap}(B)>0$ and functions \(f_1,\ldots,f_{n+1}\in{\rm Test}(X)\) such that
   $\nabla f_1,\dots,\nabla f_{n+1}$ are linearly independent on $B$. Then by \autoref{lemma:matrixlin} the $(n+1)\times (n+1)$ symmetric matrix 
	\[
		A_{ij}(y)\coloneqq\nabla f_i(y)\cdot\nabla f_j(y)
	\]
	is invertible for $\rm Cap$-a.e.\ $y\in B$.
	\\
	Since $\Per_E\ll \mathrm{Cap}$,
	and since for any test functions $f,g\in\Test(X)$ it holds that $\abs{\nabla f}^2$ and $\nabla f\cdot\nabla g$ are in $H^{1,2}$, thanks to \autoref{thm:caprepr} there exists $x\in B$ such that the following properties hold:
	\begin{itemize}
		\item[i)] the limits 
		\[\begin{split}
			\abs{\nabla f_i}^2(x)&:=\lim_{r\to 0}\fint_{B_r(x)} \abs{\nabla f_i}^2(y)\di\meas(y)\,,\\
			\nabla f_i(x)\cdot\nabla f_j(x)&:=\lim_{r\to 0}\fint_{B_r(x)}\nabla f_i(y)\cdot\nabla f_j(y)\di\meas(y)
		\end{split}\]
		exist for any $i,j\in \{1,\dots, n + 1\}$;
		\item[ii)] it holds
		\[
			\lim_{r\to 0}\fint_{B_r(x)} \abs{\abs{\nabla f_i}(y)-\abs{\nabla f_i}(x)}^2\di\meas(y)=0\, 
		\]
		and 
		\[
			\lim_{r\to 0}\fint_{B_r(x)} \abs{\abs{\nabla f_i\cdot \nabla f_j}(y)-\abs{\nabla f_i\cdot\nabla f_j}(x)}^2\di\meas(y)=0\, ,
		\]
		for any $i,j\in \{1,\dots, n+1\}$;
		\item[iii)] the matrix $A_{ij}(x)=\nabla f_i(x)\cdot\nabla f_j(x)$ is invertible;
		\item[iv)] 
		\[
			\lim_{r\to 0}r^2\fint_{B_r(x)}\left(\Delta f_i(y)\right)^2\di\meas(y)=0\, ,
		\]
		for any $i=1,\dots, n + 1$.
	\end{itemize}
	
	A by now classical argument implies then that, up to extraction of a subsequence that we do not relabel, for any sequence $r_m\to 0$, the functions
	\[
		f^m_i(y):=\left(f_i(y)-f_i(x)\right)/r_m
	\]
	considered on the scaled pointed m.m.s.\ $(X,\dist/r_m,\meas_{m},x)$ converge locally uniformly and in $H^{1,2}_{\loc}$ to Lipschitz harmonic functions $g_i: \setR^n \to\setR$, where $(\setR^n, \dist_{eucl}, \Leb^n)$ is the tangent cone of $(X,\dist,\meas)$ at $x$.  
	Moreover, the functions $g_i$ have constant slopes and 
	\[
		\nabla g_i\cdot \nabla g_j=A_{ij}(x)\, ,\quad\text{$\meas_Z$-a.e.\ on $Z$}\, .
	\]
	It is now easy to check that this is in contradiction with the fact that the dimension of $(\setR^n, \dist_{eucl}, \Leb^n)$ is $n$, see \cite{AntonelliBrueSemola19}. Indeed the functions $g_i$ are linearly independent splitting functions.
	This proves the first part of the statement, while the last part follows from the first one by recalling that
	\(L^0_E(TX)=L^0_{\rm Cap}(TX)_{|D\chi_E|}\), and by taking \autoref{lem:local_dim_drops} and \autoref{lem:equiv_dim} into account.
\end{proof}

\begin{remark}
By arguing as we did in the proof of \autoref{prop:indtest} and using the results in \cite{Kitabeppu19},
one can prove that the local dimension of \(L^0_{\rm Cap}(TX)\) on the whole \(X\) does not exceed \(n\).
This upper dimensional bound cannot be improved to an equality: considering for instance the unit interval \(X=[0,1]\), we see that the capacitary
tangent module vanishes on the boundary \(\{0,1\}\), but the latter has positive capacity.
\end{remark}

\begin{proof}[Proof of \autoref{prop:harmcoord}]
	We start by recalling that in \cite{BruePasqualettoSemola19} (see also \cite[Chapter 5]{Semola20} for a different formulation of the result) the following statement has been proved. For any set of finite perimeter $E\subset X$ and for any $\eps>0$ there exist countably many $k$-tuples of harmonic functions $(u^j_i)\, :\, B_{r_j}(x_j)\to\setR^{k_j}$, where $i\in \{1,\dots, k_j\}$, $j\in \setN$ and $k_j\le n$ such that, for $\abs{D\chi_E}$-a.e.\ $x\in \mathcal{F}_k E$, for some $j\in\setN$ such that $k_j = k$ it holds that 
	\begin{equation}\label{eq:almostorto}
		\lim_{r\to 0}\fint_{B_r(x)}\abs{\nabla u^j_a\cdot\nabla u ^j_b-\delta_{ab}}\di\meas\le \eps\, .
	\end{equation}
	Since $|D\chi_E|(\mathcal{F}_k)=0$ for any $k\neq n$, as a consequence of \autoref{cor:repr}, we just consider the case $k=n$. Let $A:=A_j$ be the set of those $x\in \mathcal{F}_n E$ such that the above condition holds for $(u_i):=(u^j_i):B_{r_j}(x_j)\to\setR^n$. It is sufficient to prove that the statement holds for $\abs{D\chi_E}$-a.e.\ $x\in A$. The conclusion will follow since the $A_j$'s cover $\mathcal{F}_n E$ up to an $\abs{D\chi_E}$-negligible set.
	
	\medskip
	
	Thanks to \autoref{lemma:gradharm}, we can also restrict $A$ to $A'$ making the further requirement that
	\begin{equation}\label{eq:lebm}
		\lim_{r\to 0}\fint_{B_r(x)}\abs{\nabla u_i(y)\cdot\nabla u_j(y)-\nabla u_i(x)\cdot\nabla u_j(x)}\di\meas(y)=0\, ,
	\end{equation} 
	for any $x\in A'$.\\
	By a Lebesgue point argument w.r.t.\ the asymptotically doubling measure $\abs{D\chi_E}$ (see \autoref{prop:asymdoub} and the discussion below), we can also assume that 
	\begin{equation}\label{eq:lebper}
		\lim_{r\to 0}\fint_{B_r(x)}\abs{\nabla u_i(y)\cdot\nabla u_j(y)-\nabla u_i(x)\cdot\nabla u_j(x)}\di\abs{D\chi_E}(y)=0\,, 
	\end{equation}
	for any $x\in A'$.\\
	Above, $\nabla u_i\cdot\nabla u_j$ is pointwise defined by \autoref{lemma:gradharm}. 
	
	The combination of \eqref{eq:almostorto} and \eqref{eq:lebm} yields that
	\begin{equation}\label{eq:almortpoint}
		\abs{\nabla u_i(x)\cdot\nabla u_j(x)-\delta_{ij}}\le \eps\, \quad\text{for any $x\in A'$} \, .
	\end{equation}
	With an additional Lebesgue point argument (again w.r.t.\ the asymptotically doubling measure $\abs{D\chi_E}$) we can restrict to a set $A''\subset A'$ considering only those points for which, in addition to the properties above, the limits in \eqref{eq:exsisttracenu} exist and \eqref{eq:reinflebper} holds.
	\medskip
	
	Observe that, for any $x\in A''$ there exists an invertible $n\times n$ matrix $A_x$ such that, setting $u':=A_xu:B_r(x)\to\setR^n$, $u'$ verifies \eqref{eq:ortomper}.\\
	Moreover, the limit in \eqref{eq:exsisttracenu} still exists for $u'$ and \eqref{eq:reinflebper} holds. 
	
	We are left to verify that, if $\nu_x\in\setR^n$ is defined according to this procedure, then $\abs{\nu_x}_{\setR^n}=1$ for $\abs{D\chi_E}$-a.e.\ $x\in A''$.
	\medskip

	We claim that $\{\nabla u_i\, :\, i=1,\dots,n\}$ form a basis of $L^2_E(TX)$ on $A''$. This property follows from two observations. The first one is that the dimension of $L^2_E(TX)$ is not greater than $n$, by \autoref{prop:indtest}. 
	The second one is that they are linearly independent, in the quantitative sense of \eqref{eq:almortpoint}.\\
	We claim that it is possible to find functions $a_1\dots,a_n\in L^{\infty}(\abs{D\chi_E})$ such that 
	\begin{equation}\label{eq:writeinbasis}
		\nu_E=\sum_{i=1}^n a_i\nabla u_i\, ,\quad \text{$\abs{D\chi_E}$-a.e.\ on $A''$}\, .
	\end{equation} 
	The existence of the coefficients follows in turn from the fact that $\{\nabla u_i\, : \, i=1, \ldots, n\}$ forms a basis, while the boundedness follows from the quantitative linear independence \eqref{eq:almortpoint}. See also the proof of \cite[Proposition 1.4.5]{Gigli18} for similar constructions.
	
	Let us observe that, since $\abs{\nu_E}_{\mathbb{R}^n}=1$ holds $\abs{D\chi_E}$-a.e.\ by \autoref{thm:GGRCDsmooth},
	\[
		\abs{\sum_{i=1}^n a_i\nabla u_i}^2=\sum_{1\le i, j\le n}a_ia_j\nabla u_i\cdot\nabla u_j=1\, , \quad\text{$\abs{D\chi_E}$-a.e.\ on $A''$}\, .
	\]
	Moreover, at $\abs{D\chi_E}$-a.e.\ point in $A''$, the coefficients $a_i$ as in \eqref{eq:writeinbasis} are uniquely determined by the values of
	\[
		\nu_E\cdot \nabla u_i\, \quad\text{and }\,\,  \nabla u_i\cdot\nabla u_j\, ,
	\]
	for any $1\le i,j\le n$.
	
	\medskip
	
	If $A_x$ as above denotes a matrix which transforms the $\nabla u_i(x)$ into an orthonormal basis $\nabla u_i'(x)$ at $x$, i.e.\ such that 
	\[
		\nabla u_i'(x)\cdot \nabla u'_j(x)=\delta_{ij}\, ,
	\]
	then, denoting by $a_i'$ the coefficients such that
	\[
		\nu_E(x)=\sum_{i=1}^n a_i'\nabla u_i'(x)\, ,
	\]
	the following hold:
	\[
		a_i'=\nu_E\cdot\nabla u_i'=\lim_{r\to 0}\fint_{B_r(x)}\nu_E\cdot\nabla u_i'\di\abs{D\chi_E}\, 
	\]
	and 
	\[
		\sum_{i=1}^{n}(a_i')^2=1\, ,
	\]
	concluding the proof of the proposition.
\end{proof}

\section{Cut and paste of sets with finite perimeter}
\label{sec:cutpaste}

Given two sets of finite perimeter $E,F\subset X$ it is simple to check that $E\cap F$ and $E\setminus F$ are of finite perimeter as well. In several applications it is relevant to characterize the perimeter measure and the interior normal of $E\cap F$ and $E\setminus F$ in terms of those of $E$ and $F$. The main achievement of this section is to extend classical results in this direction (see \cite[Theorem 16.3]{Maggi12}) to the setting of $\RCD$ spaces.
Recently, with the growing interest towards Geometric Measure Theory on metric measure spaces, and in particular on $\RCD$ spaces, these tools have become relevant also in this framework, see \cite{AntonelliPasqualettoPozzetta21,MondinoSemola21}.

Before stating our main result \autoref{thm:cutandpaste} we need to prove a Federer-type characterization for sets of finite perimeter on $\RCD$ spaces (see  \autoref{prop:federertype}) and a finer characterization of blow-ups at boundary points where {\it good coordinates} exist (see \autoref{prop:blowupsandharmcoord}).

\subsection{Federer-type characterization of sets with finite perimeter}

Let us recall a mild regularity result for sets of finite perimeter which follows again from \cite{BruePasqualettoSemola19}. It can be considered as a counterpart tailored for this framework of the Euclidean Federer-type characterization of sets of finite perimeter, see \cite[Theorem 16.2]{Maggi12}.

\begin{definition}\label{def:denstheta}
In order to ease the notation, given a set of finite perimeter $E\subset X$ and $x\in X$ we shall denote by
\[
\theta(E,x):=\lim_{r\to 0}\frac{\meas(E\cap B_r(x))}{\meas(B_r(x))}\, ,
\]
whenever the limit exists.
\end{definition}

\begin{proposition}\label{prop:federertype}
Let $(X,\dist,\meas)$ be an $\RCD(K,N)$ metric measure space for some $K\in\setR$ and $1\le N<\infty$ and let $E\subset X$ be a set of finite perimeter with finite measure. Then the following hold:
\begin{itemize}
\item[i)] For $\haus^{h}$-a.e.\ $x\in X$ it holds
\[
\theta(E,x)\in \left\lbrace 0,\frac{1}{2},1 \right\rbrace \, .
\]
Moreover, up to an $\haus^{h}$-negligible set it holds
\[
\mathcal{F}E=\left\lbrace x\in E\, :\, \theta(E,x)=\frac{1}{2}\right\rbrace\, .
\]
\item[ii)] For $\haus^{h}$-a.e.\ $x\in X$ it holds
\[
\lim_{t\downarrow 0}P_t\chi_E(x)\in \left\lbrace 0,\frac{1}{2},1\right\rbrace\, .
\]
Moreover, up to a $\haus^{h}$-negligible set it holds
\[
\mathcal{F}E=\left\lbrace x\in E\, :\, \lim_{t\downarrow 0}P_t\chi_E(x)=\frac{1}{2}\right\rbrace\, .
\]
\end{itemize}

\end{proposition}

\begin{proof}
Observe that in this framework, the perimeter measure coincides, up to constant, with the restriction of $\haus^{h}$ to the reduced boundary $\mathcal{F}E$, see \autoref{cor:repr} above. 
\medskip

Combining the outcomes of \cite{AmbrosioBrueSemola19,BruePasqualettoSemola19} and the general theory of sets of finite perimeter on PI spaces (see \cite{Am01,Am02}), we also know that the perimeter measure $\Per_E$ and $\haus^{h}$ are mutually absolutely continuous on $\partial^*E$, where
\begin{equation}\label{eq:essentialboundary}
\partial^*E:=\left\lbrace x\in X\, :\, \limsup_{r\to 0}\frac{\meas(B_r(x)\cap E)}{\meas(B_r(x))}>0\, \quad\text{and}\, \quad  \limsup_{r\to 0}\frac{\meas(B_r(x)\setminus  E)}{\meas(B_r(x))}>0\right\rbrace\, ,
\end{equation}
and, by De Giorgi's theorem for sets of finite perimeter on $\RCD(K,N)$ spaces \autoref{thm:DGRCD},
\begin{equation}\label{eq:density1/2}
\lim_{r\to 0}\frac{\meas(B_r(x)\cap E)}{\meas(B_r(x))}=\frac{1}{2}\, ,\quad\text{for $\Per_E$-a.e.\ $x\in X$}\, .
\end{equation}
In particular, \eqref{eq:density1/2} holds for $\haus^{h}$-a.e.\ $x\in \partial^*E$.

Observe now that
\[
X\setminus\partial^*E=\left\lbrace x\in X\, :\, \lim_{r\to 0}\frac{\meas(B_r(x)\cap E)}{\meas(B_r(x))}=0\right\rbrace  \bigcup \left\lbrace x\in X\, :\, \lim_{r\to 0}\frac{\meas(B_r(x)\cap E)}{\meas(B_r(x))}=1 \right\rbrace\, .
\]
\smallskip

The second part of the statement can be proved with an analogous strategy, relying on the stability of the heat flow together with a blow-up procedure, building on \autoref{thm:DGRCD}. We refer to \cite[Proposition 4.39, Corollary 5.21]{Semola20} for a more detailed argument.
\end{proof}

\begin{definition}
Given a set of finite perimeter $E\subset X$ we set
\[
E^{(t)}:=\{x\in X\, :\, \theta(E,x) =t\}\, ,
\]
where we recall that the density $\theta(E,\cdot)$ has been introduced in \autoref{def:denstheta}.
\end{definition}

\begin{remark}
By \autoref{prop:federertype} it holds
\begin{equation}\label{eq:tripartition}
X=E^{(1)}\cup E^{(1/2)}\cup E^{(0)}\, ,
\end{equation}
up to an $\haus^{h}$-negligible set.
\end{remark}

\begin{remark}\label{remark:heatdensity vs density}
	By ii) in \autoref{prop:federertype} it holds
	\[
		X = \left\lbrace \lim_{t\to 0}P_t\chi_E = 0 \right\rbrace 
		\cup 	
		\left\lbrace \lim_{t\to 0}P_t\chi_E = 1/2 \right\rbrace 
		\cup 	
		\left\lbrace
		 \lim_{t\to 0}P_t\chi_E = 1 
		\right\rbrace \, ,
	\]
	up to a $\haus^{h}$-negligible set. 
	A simple blow-up argument shows that
	\[
		E^{(0)} = \left\lbrace \lim_{t\to 0}P_t\chi_E = 0 \right\rbrace \, ,
		\quad
		E^{(1/2)} = \left\lbrace \lim_{t\to 0}P_t\chi_E = 1/2 \right\rbrace \, ,
		\quad
		E^{(1)} = \left\lbrace \lim_{t\to 0}P_t\chi_E = 1 \right\rbrace \, ,
	\]
		up to a $\haus^{h}$-negligible set.	
\end{remark}

In the following we shall adopt the notation $M\sim N $ to indicate that two Borel sets coincide up to $\haus^{h}$ negligible sets, i.e.\ $\haus^{h}(M\Delta N)=0$.

It follows from the discussion above that, for any Borel set $M\subset X$, it holds
\[
M\sim (M\cap E^{(1)})\cup (M\cap E^{(0)})\cup (M\cap E^{(1/2)})\, .
\]

\subsection{Good coordinates}
In this section we introduce the notion of {\it good coordinates at a boundary point $x$} and we employ them to give a pointwise notion of interior normal. The latter, very much in the Euclidean spirit, will be used to characterize blow-up.

\begin{definition}\label{def:goodcord}
	Let $(X,\dist,\meas)$ be an $\RCD(K,N)$ metric measure space for some $K\in\setR$ and $1\le N<\infty$ with essential dimension $1\le n\le N$ and let $E\subset X$ be a set of finite perimeter. Then, for any $x\in \mathcal{F}_n E$, any $n$-tuple of harmonic functions $(u_i):B_{r_x}(x)\to\setR^n$ satisfying the following properties is called a system of \emph{good coordinates} for $E$ at $x$.
	
	\begin{itemize}
		
		\item[(i)] For any $i,j\in \{ 1,\dots, n \}$,
		\begin{equation*}
			\lim_{r\to 0}\fint_{B_r(x)}\abs{\nabla u_i\cdot\nabla u_j-\delta_{ij}}\di\meas=\lim_{r\to 0}\fint_{B_r(x)}\abs{\nabla u_i\cdot\nabla u_j-\delta_{ij}}\di\abs{D\chi_E}=0\, .
		\end{equation*}
		
		\item[(ii)] For any $i\in\{1,\dots, n\}$ the following limits exist:
		\begin{equation*}
			\nu_i(x):=\lim_{r\to 0}\fint_{B_r(x)}\nu_E\cdot\nabla u_i\di\abs{D\chi_E},
		\end{equation*}
		\begin{equation*}
			\lim_{r\to 0}\fint_{B_r(x)}\abs{\nu_i(x)-\nu_E\cdot\nabla u_i}\di\abs{D\chi_E}=0 \, .
		\end{equation*}
	\end{itemize}
\end{definition}

As a consequence of \autoref{prop:harmcoord}, good coordinates at $x\in \mathcal{F}_n E$ exist for $|D\chi_E|$-a.e.\ point $x$.
Moreover, setting $\nu(x):=(\nu_1(x),\dots,\nu_n(x))$, it holds $\abs{\nu(x)}_{\setR^n}=1$.

\begin{remark}\label{remark:doublegoodcoordinates}
	Given two sets of finite perimeter $E,F\subset X$, by adapting the argument of the proof of \autoref{prop:harmcoord}, one can show the existence of $(u_i): B_r(x)\to \setR^n$ that are good coordinates at $x$ for both $E$ and $F$, for $\haus^h$-a.e.\ $x\in \mathcal{F} E \cap \mathcal{F} F$.
\end{remark}

\begin{proposition}\label{prop:blowupsandharmcoord}
Let $(X,\dist,\meas)$ be an $\RCD(K,N)$ metric measure space for some $K\in\setR$ and $1\le N<\infty$, and let $1\le n\le N$ be its essential dimension. Let $E\subset X$ be a set of finite perimeter. Then, for $\abs{D\chi_E}$-a.e.\ $x\in X$ and for any set of good coordinates $(u_i):B_{r_x}(x)\to\setR^n$, if $\nu\in\setR^n$ is given by \autoref{prop:harmcoord}, then the following holds. 
If the coordinates $(x_i)$ on the tangent space to $(X,\dist,\meas)$ at $x$ (which is Euclidean) are chosen so that the harmonic functions $(u_i)$, when rescaled properly, converge to $(x_i):\setR^n\to\setR^n$, then the blow-up of $E$ at $x$ in the sense of sets of finite perimeter is 
\[
H_{\nu}:=\{y\in\setR^n\, :\, y\cdot\nu \ge 0\}\, .
\]
\end{proposition}

\begin{proof}
Up to rotating via a orthonormal matrix the harmonic good coordinates, we can assume without loss of generality that the following holds: for any $i\in\{1,\dots, n-1\}$
\begin{equation}\label{eq:ortobound}
\lim_{r\to 0}\fint_{B_r(x)}\abs{\nu_E\cdot\nabla u_i}\di\abs{D\chi_E}=0\,
\end{equation}
and 
\begin{equation}\label{eq:unitpoints}
\lim_{r\to 0}\fint_{B_r(x)}\abs{1-\nu_E\cdot\nabla u_k}\di\abs{D\chi_E}=0\, .
\end{equation}
Moreover, the blow-up of $E$ at $x$ is an $n$-dimensional Euclidean half-space and the $(u_i)$, when scaled properly, converge to the coordinate functions $(x_i):\setR^n\to\setR^n$.

Under these assumptions we wish to argue as in Step 1 in the proof of \cite[Proposition 4.7]{BruePasqualettoSemola19} to prove that the blow-up of $E$, read in these coordinates is
\[
\{y\in\setR^n\, :\, y_n  \ge  0\}\, .
\]

The arguments in \cite{BruePasqualettoSemola19} prove that, by \eqref{eq:ortobound}, the blow-up of $E$ at $x$ can either be 
\[
\{y\in\setR^n\, :\, y_n\le 0\}\, ,
\quad \text{or} \quad
\{y\in\setR^n\, :\, y_n\ge  0\}\, .
\]
We claim that \eqref{eq:unitpoints} is sufficient to exclude the first possibility. Let us denote by $H$ the blow-up of $E$ at $x$ (along a sequence of scalings with the properties above). By \eqref{eq:unitpoints}, for any smooth function with compact support $\phi:\setR^n\to\setR$, 
\begin{equation}\label{eq:limitintbyparts}
\int_H\frac{\partial \phi}{\partial x_n}\di\haus^n = -\int _{\mathcal{F}H}\phi\di\haus^{n-1}\, 
\end{equation}
which proves that $H=\{y\, :\, y_n \ge  0\}$.

\medskip

In order to check \eqref{eq:limitintbyparts}, arguing as in \cite{BruePasqualettoSemola19} we just need to pass to the limit the Gauss-Green integration by parts formulae along the sequence of scalings of $E$ converging to the blow-up.\\
Suppose without loss of generality that $\phi$ is compactly supported in $B_1(0^n)\subset \setR^n$. Then, thanks to \cite{AmbrosioHonda17} we find a sequence of uniformly bounded and uniformly Lipschitz functions $\phi_m:X\to\setR$ compactly supported in $B_{r_m}(x)$ such that, when considered along the sequence of scaled spaces $X_m:=(X,r_m^{-1}\dist,\meas_m,x)$, they converge strongly in $H^{1,2}$ to $\phi$ (see \cite[Definition 5.2]{AmbrosioHonda17}). Let $E_m$ be $E\subset X_m$ and $u_i^m$ be the good coordinates scaled by $u^m_i(y)=(u_i(y)-u_i(x))/r_m$. Then we can pass to the limit the Gauss-Green formulae
\[
\int_{E_m}\div(\phi_m\nabla u_n^m)\di\meas_m=-\int_{\mathcal{F}E_m}\phi_m\nabla u_n^m\cdot\nu_{E_m}\di\abs{D\chi_{E_m}}\, ,
\]
which are obtained by scaling from the Gauss-Green integration by parts formula for $E$, to obtain \eqref{eq:limitintbyparts}. Indeed the left-hand sides converge to the left-hand side thanks to the strong $H^{1,2}$-convergence of $\phi_m$ to $\phi$. The right-hand sides instead can be written as
\[
\int_{\mathcal{F}E_m}\phi_m\nabla u_n^m\cdot\nu_{E_m}\di\abs{D\chi_{E_m}}=\int_{\mathcal{F}E_m}\phi_m\di\abs{D\chi_{E_m}}+\int_{\mathcal{F}E_m}\phi_m\left(\nabla u_n^m\cdot\nu_{E_m}-1\right)\di\abs{D\chi_{E_m}}\, ,
\]
where the second contribution converges to $0$ as $m\to\infty$ thanks to \eqref{eq:unitpoints}, since
\begin{align*}
\abs{\int_{\mathcal{F}E_m}\phi_m\left(\nabla u_n^m\cdot\nu_{E_m}-1\right)\di\abs{D\chi_{E_m}}}\le& \int _{\mathcal{F}E_m}\abs{\phi_m}\abs{1-\nu_{E_{m}}\cdot\nabla u^m_n}\di\abs{D\chi_{E_m}}\\
\le &\max_{B_{r_m}(x)}\abs{\phi_m}\fint_{B_{r_m}(x)}\abs{1-\nu_E\cdot\nabla u_n}\di\abs{D\chi_E}\to 0\, ,
\end{align*}
as $m\to\infty$.
\end{proof}

\subsection{Main result}
Below we consider the behaviour of the unit normal vector field and of the Gauss-Green formula with respect to natural cut and paste operations with sets of finite perimeter.\\
We refer to \cite[Theorem 16.13]{Maggi12} for the analogous statement for sets of finite perimeter on $\setR^n$. 

\begin{definition}
	    Let $(X,\dist,\meas)$ be an $\RCD(K,N)$ space with essential dimension $1\le n \le N$.
		Let $E,F\subset X$ be sets of finite perimeter. We define
		\[
			\{\nu_E = \nu_F \} : = 	\{\mathcal{F} E\cap \mathcal{F} F\, : \, \nu_E = \nu_F \} \, ,
		\]
	    where $\{\mathcal{F} E\cap \mathcal{F} F\, : \, \nu_E = \nu_F \}$ is the set of those $x\in \mathcal{F}_n E\cap \mathcal{F}_n F$ such that there exist good coordinates $u: B_r(x)\to \setR^n$ for both $E$ and $F$, such that $\nu_E(x) = \nu_F(x)$.
	    
	    The set $\{\nu_E = - \nu_F \}$ is defined analogously.
\end{definition}

We briefly comment on the well-posedness of the previous definition. \autoref{remark:doublegoodcoordinates} ensures that
at \(\haus^h\)-a.e.\ \(x\in\mathcal F_n E\cap\mathcal F_n F\) one can find good coordinates for both \(E\) and \(F\) at \(x\) simultaneously.
Moreover, the fact that the set \(\{\nu_E=\nu_F\}\) is independent (up to \(\haus^h\)-null sets) of the chosen good coordinates
is a consequence of the following equivalent characterization of \(\{\nu_E=\nu_F\}\).
\begin{remark}\label{rmk:equiv_nu_E=nu_F}
Given that \(|D\chi_E|,|D\chi_F|\ll{\rm Cap}\), we have that \(\mu_{E,F}\coloneqq|D\chi_E|+|D\chi_F|\ll{\rm Cap}\).
Recall that \(L^0_E(TX)=L^0_{\rm Cap}(TX)_{|D\chi_E|}\) and \(L^0_F(TX)=L^0_{\rm Cap}(TX)_{|D\chi_F|}\),
where we are using the notation for quotient modules introduced before \autoref{lem:local_dim_drops}.
For brevity, let us also set
\[
L^0_{E,F}(TX)\coloneqq L^0_{\rm Cap}(TX)_{\mu_{E,F}}\, .
\]
Given that \(|D\chi_E|,|D\chi_F|\leq\mu_{E,F}\), it is clear that
\[
L^0_E(TX)=L^0_{E,F}(TX)_{|D\chi_E|}\, ,\quad
L^0_F(TX)=L^0_{E,F}(TX)_{|D\chi_F|}\, .
\]
Moreover, it is easy to check that the projection map
\(\pi_{|D\chi_E|}\colon L^0_{E,F}(TX)\to L^0_E(TX)\)
is bijective if restricted to \(L^0_{E,F}(TX)|_{\mathcal F E}\),
so we denote by \(i_E\colon L^0_E(TX)\to L^0_{E,F}(TX)\)
its partial inverse. Similarly, we can define the map
\(i_F\colon L^0_F(TX)\to L^0_{E,F}(TX)\). Therefore, it holds that
\[
\{\nu_E=\nu_F\}\sim\big\{x\in\mathcal F E\cap\mathcal F F\,:\,|i_E(\nu_E)-i_F(\nu_F)|(x)=0\big\},
\]
as one can easily check with an argument analogous to the one employed in the proof of \autoref{thm:cutandpaste} below.
\end{remark}

\medskip

For the sake of notation, we shall indicate by $\mu_E=\nu_E\cdot \Per_E$ the \emph{Gauss-Green measure}, where we understand that 
\[
\int _E\div v\di\meas=-\int v\di\mu_E\, ,
\]
for any set of finite perimeter $E$ and any vector field $v$ verifying the assumptions of \autoref{thm:GGRCDsmooth}.

\begin{theorem}\label{thm:cutandpaste}
Let $(X,\dist,\meas)$ be an $\RCD(K,N)$ metric measure space for some $K\in\setR$ and $1\le N<\infty$. 
Let $E,F\subset X$ be sets of finite perimeter. 
Then $E\cap F$, $E\cup F$ and $E\setminus F$ are sets of finite perimeter and
\begin{subequations}\begin{align}\label{eq:GGEcapF}
\mu_{E\cap F}&=\mu_E\res F^{(1)}+\mu_F\res E^{(1)}+\nu_E\haus^h\res\{\nu_E=\nu_F\}\, ,\\
\label{eq:GGEcupF}
\mu_{E\cup F}&=\mu_E\res F^{(0)} + \mu_F\res E^{(0)}+\nu_E\haus^h\res\{\nu_E=\nu_F\}\, ,\\
\label{eq:GGEsetminusF}
\mu_{E\setminus F}&=\mu_E\res F^{(0)}-\mu_F\res E^{(1)}+\nu_E\haus^h\res \{\nu_E=-\nu_F\}\, .
\end{align}\end{subequations}
\end{theorem}

Let us clarify the meaning of \eqref{eq:GGEcapF}, the meaning of \eqref{eq:GGEcupF} and \eqref{eq:GGEsetminusF} can be deduced by analogy.
With this notation we mean that, for any vector field $v\in H^{1,2}_C(TX)\cap D(\div )$ such that $\abs{v}\in L^{\infty}(\meas)$,
\begin{align*}
\int _{E\cap F}\div v\di\meas
=&-\int_{F^{(1)}} \langle \mathrm{tr}_E v,\nu_E \rangle \di\Per_E-\int_{E^{(1)}} \langle \mathrm{tr}_F v,\nu_F \rangle \di\Per_F\\
&-\int_{\{\nu_E = \nu_F\}} \langle \mathrm{tr}_E v,\nu_E \rangle \di\haus^h\, .
\end{align*}

\begin{proof}
If $E,F\subset X$ are sets of finite perimeter, then (on general ambient metric measure spaces) $E\cap F$, $E\cup F$ and $E\setminus F$ are sets of finite perimeter. 
\medskip

Let us make a preliminary observation. By decomposing, for any $x\in X$ and for any $r>0$, $(E\cup F)\cap B_r(x)$ into the disjoint union
\[
(E\cup F)\cap B_r(x)= \left((E\setminus F)\cap B_r(x)\right)\cup \left((F\setminus E)\cap B_r(x)\right)\cup \left((E\cap F)\cap B_r(x)\right)\, ,
\]
we can easily infer that, at any point of existence of the densities $\theta(\cdot,\cdot)$ (see \autoref{def:denstheta} for the relevant notion)
\begin{equation}\label{eq:estdensau}
\max\{\theta(E,x),\theta(F,x)\}\le \theta(E\cup F,x)\le \theta(E,x)+\theta(F,x)-\theta(E\cap F,x)\, .
\end{equation} 
\medskip
We are going to prove \eqref{eq:GGEcapF}, the proofs of the other statements being completely analogous. 
First, relying on \eqref{eq:estdensau} and arguing as in the proof of \cite[Theorem 16.3]{Maggi12}, we obtain that
\begin{equation}\label{eq:decdens}
(E\cap F)^{(1/2)}\sim \left(F^{(1)}\cap E^{(1/2)}\right)\cup \left(E^{(1)}\cap F^{(1/2)}\right)\cup\left((E\cap F)^{(1/2)}\cap E^{(1/2)}\cap F^{(1/2)}\right)\, 
\end{equation}
and that the three sets at the right-hand side have mutually $\haus^{h}$-negligible intersections.

Let us see now how to prove the representation formula \eqref{eq:GGEcapF} for the Gauss-Green measure of $E\cap F$.

Let us briefly recall the strategy in the Euclidean setting. Given \eqref{eq:decdens}, which identifies the reduced boundary of $E\cap F$, and De Giorgi's theorem, it remains only to determine the unit normal vector to $E\cap F$ on the different components of the reduced boundary in the decomposition. On $\setR^n$ the blow-up of a set of finite perimeter at a reduced boundary point is the half-space orthogonal to the unit normal vector. As we shall see, the combination of \autoref{prop:harmcoord}, \autoref{remark:doublegoodcoordinates} and \autoref{prop:blowupsandharmcoord} is a replacement of this fact in our framework. 
\medskip

Let us first deal with $\mu_{E\cap F}\res F^{(1)}$. We wish to prove that it coincides with $\mu_E\res F^{(1)}$.\\
Suppose by contradiction that this is not the case. Then we can find a set of positive $\haus^{h}$ measure on $(E\cap F)^{(1/2)}\cap F^{(1)}$ where the identity does not hold. In particular, applying \autoref{prop:harmcoord} \autoref{remark:doublegoodcoordinates} and \autoref{prop:blowupsandharmcoord} to both $E$ and $E\cap F$ we can find $x\in X$ such that: 
\begin{itemize}
\item[i)] $x\in F^{(1)}$ and $x\in \mathcal{F}_n E \cap \mathcal{F}_n (E\cap F)$;
\item[ii)] there exist $r>0$ and a set of good coordinates $(u_i):B_r(x)\to\setR^n$ for both $E$ and $E\cap F$ at $x$.
\end{itemize}
Moreover, by the contradiction assumption, we can suppose that
\begin{equation}\label{eq:diffnorm}
\nu_{E\cap F}(x)\neq \nu_E(x)\, ,
\end{equation}
where $\nu_{E\cap F}(x),\nu_E(x)\in\setR^n$,
\begin{equation}\label{z1}
\left(\nu_{E\cap F}(x)\right)_i=\lim_{s\to 0}\fint_{B_s(x)}\nu_{E\cap F}(y)\cdot\nabla u_i(y)\di\abs{D\chi_{E\cap F}}(y)\, ,\quad\text{for any $i=1,\dots,n$}\, ,
\end{equation}
and 
\begin{equation}\label{z2}
\left(\nu_{E}(x)\right)_i=\lim_{s\to 0}\fint_{B_s(x)}\nu_{E}(y)\cdot\nabla u_i(y)\di\abs{D\chi_{E}}(y)\, ,\quad\text{for any $i=1,\dots,n$}.
\end{equation}
Relying on \autoref{prop:blowupsandharmcoord}, by \eqref{eq:diffnorm}, $E$ and $E\cap F$ converge to different halfspaces in the $L^1_{\mathrm{loc}}$ sense along a sequence of rescaled pointed spaces $X_i:=(X,r_i^{-1}\dist,\meas^{r_i}_{x},x)$ converging to the tangent Euclidean space $(\setR^n,\dist_{eucl},\haus^n,0^n)$.\\
We claim that this yields a contradiction. Indeed, $x\in F^{(1)}$, equivalently $x$ is a point of density $1$ of $F$. This easily implies that along the converging sequence $X_i$, $\chi_F$ converge in the $L^1_{\loc}$ sense to the constant function $1$. Since
\[
\chi_{E\cap F}=\chi_E\cdot\chi_F\, 
\]
and $L^1_{\loc}$ convergence is stable under multiplication (see for instance \cite[Lemma 3.5]{AmbrosioBrueSemola19}), we infer that $E$ and $E\cap F$ have the same blow-up along the same sequence of converging scaled pointed spaces. 
This provides the sought contradiction, hence 
\[
\mu_{E\cap F}\res F^{(1)}=\mu_E\res F^{(1)}\, .
\]
In order to avoid confusion, let us comment on the previous assertion. Since isomorphic (pointed) metric measure spaces are identified when dealing with pmGH convergence, any two half-spaces in a Euclidean space should be considered the same from this perspective and this statement might sound strange. Here the subtle point is that good coordinates are providing a canonical way to parametrize the tangent. Once the limit coordinates are fixed, the half-spaces are uniquely identified (in terms of their unit normal vectors).

\medskip

The argument to prove that
\[
\mu_{E\cap F}\res E^{(1)}=\mu_F\res E^{(1)}
\]
is completely analogous, based on the fact that at a point in $E^{(1)}$, the set $E$, when considered along a sequence of scalings of the ambient space converging to a tangent, converges to the whole space. So we omit the details.  
\medskip

In order to prove that 
\[
\mu_{E\cap F}\res \{\nu_E=\nu_F\}=\nu_E\haus^{h}\res\{\nu_E=\nu_F\}\, ,
\]
we just need a slight variant of the argument used above. 
By contradiction, we can find $x\in \mathcal{F}_n(E\cap F)\cap \mathcal{F}_n E\cap \mathcal{F}_n F$ such that 
\begin{itemize}
	\item[i)] there exist $r>0$ and a set of good coordinates $(u_i):B_r(x)\to\setR^n$ for $E$ and $F$ and $E\cap F$ at $x$ such that $\nu_E(x) = \nu_F(x)$;
	\item[ii)] there exist the limits
	\[
		\left(\nu_{E\cap F}(x)\right)_i=\lim_{s\to 0}\fint_{B_s(x)}\nu_{E\cap F}(y)\cdot\nabla u_i(y)\di\abs{D\chi_{E\cap F}}(y)\, ,\quad\text{for any $i=1, \dots, n$}\, ,
	\]
	and $\nu_{E\cap F}(x) \neq \nu_E(x)$.
\end{itemize}
Denoting by $E_r$ and $F_r$, respectively, the set $E$ and $F$ in the rescaled pointed spaces $X_i:=(X,r_i^{-1}\dist,\meas^{r_i}_{x},x)$ we know that
\[
	\chi_{E_r} \to H_{\nu_E(x)} \, , \quad
	\chi_{F_r} \to H_{\nu_F(x)} \, , \quad
	\chi_{E_r\cap F_r} \to H_{\nu_{E\cap F}(x)} \, , 
\]
in the  $L^1_{\mathrm{loc}}$ topology, as a consequence of \autoref{prop:blowupsandharmcoord}. The stability of the $L^1_{\loc}$ convergence under multiplication implies that
\[
  \chi_{H_{\nu_E(x)}} \cdot  \chi_{H_{\nu_F}(x)}
  = \lim_{r\to 0}	\chi_{E_r} \cdot \chi_{F_r}
  = \lim_{r\to 0} \chi_{E_r\cap F_r}
  = \chi_{H_{\nu_{E\cap F}(x)}} \, ,
\]
which contradicts ii).
\end{proof}

\begin{corollary}\label{cor:GGmeasureincluded}
Let $(X,\dist,\meas)$ be an $\RCD(K,N)$ metric measure space and let $E\subset F\subset X$ be sets of finite perimeter. Then $\nu_E=\nu_F$ on $\mathcal{F}E\cap\mathcal{F}F$ $\haus^h$-a.e.,  and 
\[
\mu_E=\mu_E\res F^{(1)}+\nu_F\haus^{h}\res\left(\mathcal{F}E\cap\mathcal{F}F\right)\, .
\]
\end{corollary}

\begin{proof}
From \autoref{thm:cutandpaste} and the identity $E=E\cap F$ we deduce
\[
	\mu_E = \mu_{E\cap F}
	=\mu_E\res F^{(1)} + \nu_F\haus^{h}\res\{\nu_E = \nu_F\}\, .
\]
Hence, it suffices to prove that 
\[
	\mathcal{F} E\cap\mathcal{F} F \sim \{\nu_E = \nu_F\} \, .
\]
The latter follows from the same blow-up argument performed in the proof of \autoref{thm:cutandpaste}.
\end{proof}

\begin{remark}
The particular case of the constructions above when $E$ and $F$ have essentially disjoint reduced boundaries has been considered before in \cite{AntonelliPasqualettoPozzetta21}, see in particular Lemma 2.5 and Proposition 2.31 therein.
\end{remark}

\section{Gauss-Green formulae for essentially bounded divergence measure vector fields}\label{sec5}

The aim of this section is to sharpen the integration by parts formulae, introduced in \cite{BuffaComiMiranda19} on $\RCD(K,\infty)$ metric measure spaces (after the developments of the Euclidean theory in \cite{Anzellotti83,ChenTorresZiemer09,ComiPayne20}), for {\it essentially bounded divergence vector fields} on sets with finite perimeter.

\begin{definition}\label{def:essbounddivmeas}
Let $(X,\dist,\meas)$ be an $\RCD(K,N)$ metric measure space. We say that a vector field $V\in L^{\infty}(TX)$ is an essentially bounded divergence measure vector field if its distributional divergence is a finite Radon measure, that is if $\div V$ is  finite Radon measure such that, for any Lipschitz function with compact support $g:X\to\setR$, it holds
\begin{equation}\label{eq:intbypartsebdm}
\int_X g\di \div V=-\int_X \nabla g\cdot V\di\meas\, .
\end{equation}
We shall denote the class of these vector fields by $\mathcal{DM}^{\infty}(X)$ and sometimes, to ease the notation, we will abbreviate $\int g\di\div V$ as $\int g\div V$.

Analogously, for any open set $\Omega\subset X$ it is possible to introduce the space $\mathcal{DM}^{\infty}(\Omega)$ of locally essentially bounded divergence measure vector fields in $\Omega$.
\end{definition}

It turns out that, despite not being able to pointwise define a vector field with such low regularity over the reduced boundary of a set of finite perimeter, it is possible to define interior and exterior normal traces, possibly different, playing the role of the term $V\cdot\nu_E$ in the Gauss-Green formula.
The rigorous result is the following.

\begin{theorem}\label{thm:GaussGreenEssDivVect}
	Let $(X,\dist,\meas)$ be an $\RCD(K,N)$ metric measure space for some $K\in\setR$ and $1\le N<\infty$. Let $E\subset X$ be a set of finite perimeter and let $V\in\mathcal{DM}^{\infty}(X)$. Then we have the Gauss-Green integration by parts formulae: for any function $\phi\in \Lip_c(X)$ it holds
	\[\begin{split}
		\int_{E^{(1)}}\phi \div V+\int_E\nabla\phi\cdot V\di\meas&=-\int_{\mathcal{F}E}\phi\left(V\cdot\nu_E\right)_{\mathrm{int}}\di\Per_E \, ,\\
		\int_{E^{(1)}\cup\mathcal{F}E}\phi \div V+\int_E\nabla\phi\cdot V\di\meas&=-\int_{\mathcal{F}E}\phi\left(V\cdot\nu_E\right)_{\mathrm{ext}}\di\Per_E\, ,
	\end{split}\]
	where $(V\cdot \nu_E)_{_{\mathrm{int}}}$ and $(V\cdot \nu_E)_{_{\mathrm{ext}}}$ belong to $L^{\infty}(\mathcal{F}E,\Per_E)$ and satisfy
   \begin{subequations}\begin{align}\label{eq:inftyboundtraceintsharp}
	\norm{\left(V\cdot\nu_E\right)_{\mathrm{int}}}_{L^{\infty}(\mathcal{F}E,\Per_E)}&\le \norm{V}_{L^{\infty}(E,\meas)} \, ,\\
  \label{eq:inftyboundtraceextsharp}
	\norm{\left(V\cdot\nu_E\right)_{\mathrm{ext}}}_{L^{\infty}(\mathcal{F}E,\Per_E)}&\le \norm{V}_{L^{\infty}(X\setminus E,\meas)}\, .
   \end{align}\end{subequations}
\end{theorem}

 Given an essentially bounded divergence measure vector field $V\in \mathcal{DM}^{\infty}(X)$, it is proved in \cite[Section 6.5]{BuffaComiMiranda19} that there exist a subsequence $t_k \to 0$, a function $\tilde \chi_E \in L^\infty(|\div V|)$ and measures $D\chi_E(\chi_E V)$, $D\chi_E(\chi_{E^c}V)$ such that
 \[\begin{split}
 	P_{t_k}\chi_E \overset * \weakto &\,\tilde \chi_E \, 
 	\quad
 	\text{in $L^\infty(|\div V|)$} \, ,\\
 	\nabla P_{t_k} \chi_E\cdot (\chi_EV)\weakto &\,D\chi_E(\chi_E V) \, ,\\
 	\nabla P_{t_k}\chi_E\cdot (\chi_{E^c}V)\weakto &\,D\chi_E(\chi_{E^c} V)\, .
 \end{split}\]
 Notice that, a priori, $D\chi_E(\chi_E V)$, $D\chi_E(\chi_{E^c} V)$ and $\tilde \chi_E$ depend on the choice of the subsequence.
 It can be proven that $D\chi_E(\chi_E V)$ and $D\chi_E(\chi_{E^c}V)$ are both absolutely continuous w.r.t.\ $\abs{D\chi_E}$. Therefore we are entitled to consider their densities and set
 \begin{subequations}\begin{align}\label{eq:inttrace}
 	2D\chi_E(\chi_E V)&=\left(V\cdot\nu_E\right)_{\mathrm{int}}\abs{D\chi_E} \, ,\\
 \label{eq:exttrace}
 	2D\chi_E(\chi_{E^c} V)&=\left(V\cdot\nu_E\right)_{\mathrm{ext}}\abs{D\chi_E}\, ,
 \end{align}\end{subequations}
 respectively. In \cite[Proposition 6.5]{BuffaComiMiranda19} it is shown that
  \begin{subequations}\begin{align}\label{eq:inftyboundtraceintnonsharp}
 	\norm{\left(V\cdot\nu_E\right)_{\mathrm{int}}}_{L^{\infty}(\mathcal{F}E,\Per_E)}&\le 2\norm{V}_{L^{\infty}(E,\meas)} \, ,\\
 \label{eq:inftyboundtraceextnonsharp}
 	\norm{\left(V\cdot\nu_E\right)_{\mathrm{ext}}}_{L^{\infty}(\mathcal{F}E,\Per_E)}&\le 2\norm{V}_{L^{\infty}(X\setminus E,\meas)}\, .
\end{align}\end{subequations}

 In Theorem 6.22 \cite{BuffaComiMiranda19} the authors proved that, under the assumption that any weak-star limit in $L^\infty(|D\chi_E|)$ of $P_t\chi_E$ is constant, for any function $\phi\in \Lip_c(X)$ it holds
 \begin{subequations}\begin{align}\label{eq:GGformula nonsharp1}
 	\int_{\tilde E^{(1)}}\phi \div V+\int_E\nabla\phi\cdot V\di\meas&=-\int_{\mathcal{F}E}\phi\left(V\cdot\nu_E\right)_{\mathrm{int}}\di\Per_E \, ,\\
\label{eq:GGformula nonsharp2}
 	\int_{\tilde E^{(1)}\cup\tilde E^{(1/2)}}\phi \div V+\int_E\nabla\phi\cdot V\di\meas&=-\int_{\mathcal{F}E}\phi\left(V\cdot\nu_E\right)_{\mathrm{ext}}\di\Per_E\, ,
 \end{align}\end{subequations}
where $\tilde E^{(t)} : = \{ \tilde \chi_E = t \}$.
\medskip

Thanks to \autoref{prop:federertype} ii) we have a good understanding of the pointwise behaviour of the evoluted of the indicator function $\chi_E$ of a set with finite perimeter $E\subset X$ through the heat flow. As proven in \autoref{prop:federertype} ii), sets of finite perimeter on $\RCD(K,N)$ spaces have the following property. The sequence $P_t\chi_E$ has a unique weak-star limit as $t\to 0$ in $L^{\infty}(X,\abs{D\chi_E})$, and it is the constant function $1/2$. Hence we are in a position to apply Theorem 6.20 of \cite{BuffaComiMiranda19} to get \eqref{eq:GGformula nonsharp1} and \eqref{eq:GGformula nonsharp2}. This being said it becomes clear that \autoref{thm:GaussGreenEssDivVect} follows from the analysis is \cite{BuffaComiMiranda19} provided we show the following facts:
\begin{itemize}
	\item[(a)] 
	 $\tilde E^{(1)} = E^{(1)}$ and $\tilde E^{(1/2)} = \mathcal{F} E$ up to a $|\div V|$-negligible set;
	\item[(b)] 
	The bounds for the normal traces \eqref{eq:inftyboundtraceintnonsharp}, \eqref{eq:inftyboundtraceextnonsharp} hold in the improved form \eqref{eq:inftyboundtraceintsharp}, \eqref{eq:inftyboundtraceextsharp}.
\end{itemize}
The sharp bound claimed in (b) in the case of vector fields in $H^{1,2}_C$ is already contained in \autoref{thm:GGRCDsmooth}, since in \cite{BruePasqualettoSemola19} it was proved that $\abs{\nu_E}=1$ a.e.\ w.r.t.\ $\abs{D\chi_E}$. It turns out that the sharp trace bounds can be obtained also for the less regular vector fields in $\mathcal{DM}^{\infty}(X)$.
\medskip

As a preliminary step, let us observe that, as it happens in the Euclidean case, the divergence of an essentially bounded divergence measure vector field is always absolutely continuous w.r.t.\ $\haus^h$.

\begin{lemma}\label{lemma:divllHh}
Let $(X,\dist,\meas)$ be an $\RCD(K,N)$ metric measure space. Then for any vector field $V\in \mathcal{DM}^{\infty}(X)$ it holds that $|\div V| \ll \haus^h$.
\end{lemma}
\begin{proof}
	Let $x\in X$ and $r>0$ be such that $|\div V|(\partial B_r(x))=0$. We apply \eqref{eq:GGformula nonsharp1} with $E=B_r(x)$ and $\phi=1$. A simple blow-up argument shows that $B_r(x) \subset \tilde B_r(x)^{(1)} \subset \bar B_r(x)$, hence
	\begin{equation}\label{eq:div on balls}
		\abs{\div V(B_r(x))} \le \norm{V}_{L^\infty} \Per_{B_r(x)}(X) \le C(N)\haus^h(B_r(x)) \le  C(N) \frac{\meas(B_r(x))}{r} \, ,
	\end{equation}
	where we used that $\mathcal{F}B_r(x) \subset \partial B_r(x)$ and that $h(B_r(x)) = \frac{\meas(B_r(x))}{r}$ is a gauge function for $\haus^h$.
	
	Given a Borel set $E\subset X$ such that $\haus^h(E)=0$ and $\eps>0$, we consider a cover $E\subset \cup_{i\in \setN} B_{r_i}(x_i)$ such that, for any $i\in \setN$, it holds
	\begin{equation}\label{eq:cov}
		|\div V(\bar B_{r_i}(x_i))| 
		= |\div V( B_{r_i}(x_i))| \ge \frac{1}{2}|\div V|(B_{r_i}(x_i)) \, ,
		\quad
		\sum_{i\in \setN} \frac{\meas(B_{r_i}(x_i))}{r_i}< \eps \, .
	\end{equation}
	By using \eqref{eq:div on balls} and \eqref{eq:cov} we deduce
	\begin{align*}
		|\div V|(E) 
		\le \sum_{i\in \setN} |\div V|(B_{r_i}(x_i))
		\le 2 \sum_{i\in \setN} |\div V(B_{r_i}(x_i))|
		\le C(N) \sum_{i\in \setN} \frac{\meas(B_{r_i}(x_i))}{r_i}
		\le C(N) \eps \, ,
	\end{align*}
	which implies the sought conclusion.
\end{proof}

\begin{proof}[Proof of $\rm (a)$]
	From \autoref{lemma:divllHh} we know that $|\div V|\ll \haus^h$, hence ii) in \autoref{prop:federertype} and \autoref{remark:heatdensity vs density} imply that, up to a $|\div V|$-negligible set, it holds
	\[\begin{split}
		X &= \{ \lim_{t\to 0}P_t\chi_E = 0 \} \cup 	\{ \lim_{t\to 0}P_t\chi_E = 1/2 \} \cup 	\{ \lim_{t\to 0}P_t\chi_E = 1 \} \, ,\\
		E^{(0)} &= \{ \lim_{t\to 0}P_t\chi_E = 0 \} \, ,
		\quad
		E^{(1/2)} = \{ \lim_{t\to 0}P_t\chi_E = 1/2 \}\, ,
		\quad
		E^{(1)} = \{ \lim_{t\to 0}P_t\chi_E = 1 \} \, .
	\end{split}\]
	The sought conclusion follows from the fact that $E^{(1/2)} = \mathcal{F}E$ up to a $|\div V|$-negligible set as a consequence of  i) in \autoref{prop:federertype} and $|\div V|\ll \haus^h$.		
\end{proof}

\begin{proof}[Proof of $\rm (b)$]
Our goal is to prove that
\[
\norm{\left(V\cdot\nu_E\right)_{\mathrm{int}}}_{L^{\infty}(\mathcal{F}E,\Per)}\le \norm{V}_{L^{\infty}(E,\meas)}\, .
\]

In order to do so we just slightly refine the last computation in the proof of \cite[Lemma 5.2]{BuffaComiMiranda19} relying on \autoref{thm:DGRCD}. Basically, all we need to know is that the unit normal vector to the set of finite perimeter has length one, in a suitable sense, and that the density of the set of finite perimeter at reduced boundary points is $1/2$.
\medskip

It is sufficient to prove the following: for any nonnegative function $\phi\in C_c(X)$ it holds
\begin{equation}\label{eq:goaltraces}
\lim_{t\downarrow 0}\abs{\int \phi\chi_E\left(\nabla P_t\chi_E\cdot V\right)\di\meas}\le \frac{1}{2}\norm{V}_{L^{\infty}(E)}\int \phi\di\abs{D\chi_E}\, .
\end{equation}
Let us compute, following the proof of \cite[Theorem 2.4]{BruePasqualettoSemola19} and setting
\[
\nu_t:=\frac{\nabla P_t\chi_E}{P_t^*\abs{D\chi_E}}\, ,\;\;\quad\mu_t:=P_t^*\abs{D\chi_E}\meas\, ,
\]
\begin{align}\label{eq:est1}
\nonumber\lim_{t\downarrow 0}\abs{\int \phi\chi_E\left(\nabla P_t\chi_E\cdot V\right)\di\meas}=&\lim_{t\downarrow 0}\abs{\int e^{Kt}\phi\chi_E\left(\nu_t\cdot V\right)\di\mu_t}\\
\le& \lim_{t\downarrow 0}\left(\norm{\phi}_{\infty}\norm{V}_{\infty}\int_X\abs{1-e^{Kt}\abs{\nu_t}}\di\mu_t+\int_X\phi\chi_E\left(\frac{\nu_t}{\abs{\nu_t}}\cdot V\right)\di\mu_t\right)\, .
\end{align}
The first summand above tends to $0$ as $t\downarrow 0$ by \cite[Lemma 2.7]{BruePasqualettoSemola19}, see also \cite{AmbrosioBrueSemola19}. Let us deal with the second one.

We can estimate 
\begin{align}\label{eq:est2}
\nonumber\lim_{t\downarrow 0} \int_X\phi\chi_E\left(\frac{\nu_t}{\abs{\nu_t}}\cdot V\right)\di\mu_t\le& \norm{V}_{L^{\infty}(E,\meas)}\lim_{t\downarrow 0}\int \phi\chi_E\di\mu_t\\
= &  \norm{V}_{L^{\infty}(E,\meas)}\lim_{t\downarrow 0} \int P_t(\phi\chi_E)\di\abs{D\chi_E}\, .
\end{align}
We claim that 
\begin{equation}\label{eq:pointconv}
\lim_{t\downarrow 0} P_t(\phi\chi_E)(x)=\frac{1}{2}\phi(x)\, ,\quad\text{for $\abs{D\chi_E}$-a.e.\ $x$}\, .
\end{equation}
The validity of \eqref{eq:pointconv} can be easily checked relying on the continuity of $\phi$ at any reduced boundary point of $E$, with a simple variant of the argument leading to \autoref{prop:federertype} ii). Given \eqref{eq:pointconv} we can argue by the dominated convergence theorem that 
\[
\lim_{t\downarrow 0} \int P_t(\phi\chi_E)\di\abs{D\chi_E}=\frac{1}{2}\int \phi\di\abs{D\chi_E}\, .
\]
Hence, by \eqref{eq:est1}, we get \eqref{eq:goaltraces}. 
\end{proof}

In \autoref{thm:cutandpaste} we have dealt with the relationships between the Gauss-Green integration by parts formulae over two sets of finite perimeter and those over the sets obtained via elementary operations between them. Therein, the language was that of the tangent module over the boundary of a set of finite perimeter $L^2_E(TX)$ and vector fields were assumed to be sufficiently smooth to have pointwise defined representatives $\Per_E$-a.e.. Having at disposal a well behaved notion of interior/exterior normal trace over the boundary of a set of finite perimeter for any vector field which is bounded and has measure valued divergence, we would like to understand to which extent those operations are well behaved under these regularity assumptions.   

\begin{proposition}\label{prop:cutandpasteweaker}
Let $(X,\dist,\meas)$ be an $\RCD(K,N)$ metric measure space for some $K\in\setR$ and $1\le N<\infty$. Let $E,F\subset X$ be sets of (locally) finite perimeter and let $V\in\mathcal{DM}^{\infty}(X)$. Then the following relationships between normal traces hold true:
\begin{subequations}\begin{align}\label{zz0}
	\left(V\cdot \nu_{E }\right)_{\mathrm{int}}&=\left(V\cdot \nu_{F}\right)_{\mathrm{int}}\, , \quad\text{$\haus^h$-a.e.\ on $\{\nu_E = \nu_F\}$}\, ,\\
\label{zz1}
\left(V\cdot \nu_{E\cap F}\right)_{\mathrm{int}}&=\left(V\cdot \nu_{E}\right)_{\mathrm{int}}\, , \quad\text{$\Per_E$-a.e.\ on $F^{(1)}$}\, ,\\
\label{zz2}
\left(V\cdot \nu_{E\cap F}\right)_{\mathrm{int}}&=\left(V\cdot \nu_{F}\right)_{\mathrm{int}}\, , \quad\text{$\Per_F$-a.e.\ on $E^{(1)}$}\, ,\\
\label{zz3}
\left(V\cdot \nu_{E\cap F}\right)_{\mathrm{int}}&=\left(V\cdot \nu_{E}\right)_{\mathrm{int}}\, , \quad\text{$\haus^h$-a.e.\ on $\{\nu_E = \nu_F\}$}\, .
\end{align}\end{subequations}
Analogous conclusions hold for the exterior normal traces and for the interior and exterior normal traces on $E\cup F$ and on $E\setminus F$.
\end{proposition}
\begin{proof}
    Let us begin by proving \eqref{zz0}. We perform a blow-up argument similar to the one in the proof of \autoref{prop:blowupsandharmcoord} and inspired by the proof of \cite[Proposition 3.4.6]{Comi20}, dealing with the Euclidean case.
    
    We claim that $\haus^h$-a.e.\ $x\in E^{(1/2)}\cap F^{(1/2)}$ satisfies the following:
    \begin{itemize}
    	\item[(i)] 
    	\[\begin{split}
    		&\lim_{r\to 0}\fint_{B_r(x)} \abs{(V\cdot\nu_{E})_{{\mathrm{int}}}(y) - (V\cdot\nu_{E})_{{\mathrm{int}}}(x)} \di \Per_E(y)
    		= 0 \, ,\\
      	&\lim_{r\to 0}\fint_{B_r(x)} \abs{(V\cdot\nu_{F})_{{\mathrm{int}}}(y) - (V\cdot\nu_{F})_{{\mathrm{int}}}(x)} \di \Per_F(y)
    		= 0 \, ;  
    	\end{split}\]
        \item[(ii)] 
        \[
        \lim_{r\to 0} \frac{r|\div V|((E^{(1)}\cup F^{(1)})\cap B_r(x))}{\meas(B_r(x))}
        = 0.
        \]
    \end{itemize}
    The property (i) amounts to say that $x$ is a Lebesgue point for both $(V\cdot\nu_{E})_{{\mathrm{int}}}\in L^\infty(|D\chi_E|)$ and $(V\cdot\nu_{F})_{{\mathrm{int}}}\in L^\infty(|D\chi_F|)$. The latter hold for, respectively, $|D\chi_E|$ and $|D\chi_F|$ almost every point as explained after \autoref{prop:asymdoub}. The claimed conclusion follows by recalling that $|D\chi_E|$ and $|D\chi_F|$ are equivalent to the measure $\haus^h$ restricted, respectively, to $E^{(1/2)}$ and $F^{(1/2)}$.
    
    We now prove that $\haus^h$-a.e.\ $x\in E^{(1/2)}\cap F^{(1/2)}$ satisfies (ii). Assume the existence of $B\subset E^{(1/2)}\cap F^{(1/2)}$ with positive $\haus^h$-measure such that $x_i\in B$, $r_i< \delta$ and 
    \[
    \limsup_{r\to 0} \frac{r|\div V|((E^{(1)}\cup F^{(1)})\cap B_r(x))}{\meas(B_r(x))} \ge \eps > 0 \, ,
    \quad
    \text{for any $x\in B$} \, .
    \]
    We can then find a cover $B\subset \cup_{i\in \setN} B_{r_i}(x_i)$ such that 
    \[
    	0 < \haus^h(B) \le C(N) \sum_{i\in \setN} \frac{\meas(B_{r_i}(x_i))}{r_i} \, ,
    	\quad\quad
    	\frac{r_i|\div V|((E^{(1)}\cup F^{(1)})\cap B_{r_i}(x_i))}{\meas(B_{r_i}(x_i))} \ge \eps/2
    	\quad \forall \, i\in \setN \, .
    \]
    In particular,
    \[
    	|\div V|((E^{(1)}\cup F^{(1)})\cap \{ y\in X\, : \, \dist(x,y) \le \delta\,\,  \text{for some}\,\,  x\in B \})
    	\ge C(N,\eps) \haus^h(B) >0 \, ,
    \]
    letting $\delta\to 0$ we get the sought contradiction
    \begin{align*}
    	0 =& \,|\div V|((E^{(1)}\cup F^{(1)})\cap (E^{(1/2)}\cap F^{(1/2)})) \\
    	\ge &\,|\div V|((E^{(1)}\cup F^{(1)})\cap B)
    	\ge C(N,\eps) \haus^h(B) > 0 \, .
    \end{align*}
    
    \medskip

    Let us pick a point $x \in E^{(1/2)}\cap F^{(1/2)}$ such
    that there exits a system of good coordinates $(u_i): B_r(x) \to \setR^n$ for both $E$ and $F$ such that $\nu_E(x) = \nu_F(x)$ and the properties (i) and (ii) hold. 
    Fix a function $\phi\in C^\infty(\setR^n)$ compactly supported in $B_1(0^n)\subset \setR^n$. Then, thanks to \cite[Lemma 2.10]{AmbrosioHonda18} we find a sequence of uniformly bounded and uniformly Lipschitz functions $\phi_m:X\to\setR$ compactly supported in $B_{r_m}(x)$ such that, when considered along the sequence of scaled spaces $X_m:=(X,r_m^{-1}\dist,\meas_m,x)$ with $\meas_m=\meas/\meas(B_{r_m}(x))$, they converge strongly in $H^{1,2}$ to $\phi$ (see \cite[Definition 5.2]{AmbrosioHonda17}).\\ 
    Let $E_m$ and $F_m$ be $E, F\subset X_m$ and $u_i^m$ be the good coordinates scaled by $u^m_i(y)=(u_i(y)-u_i(x))/r_m$. We can assume that $u^m_i \to x_i$, the $i$-th coordinate of $\setR^n$, in $H^{1,2}$ and, as a consequence of \autoref{prop:blowupsandharmcoord}, $E_m\to H_{\nu_E(x)}$ and $F_m\to H_{\nu_F(x)}$ in $L^1_{\text{\loc}}$.

   By \autoref{thm:GaussGreenEssDivVect} it holds
   \begin{equation}\label{z8}
   	\int_{E^{(1)}}\phi_m \div V+\int_E \nabla\phi_m \cdot V\di\meas
   	=-\int_{\mathcal{F}E}\phi_m \left(V\cdot\nu_E\right)_{\mathrm{int}}\di\Per_E \, ,
   \end{equation}
    \begin{equation}\label{z9}
    	\int_{F^{(1)}}\phi_m \div V+\int_F \nabla\phi_m\cdot V\di\meas
    	=-\int_{\mathcal{F}F}\phi_m\left(V\cdot\nu_F\right)_{\mathrm{int}}\di\Per_F \, .
    \end{equation}
    Fix $\eps>0$ and observe that
    \begin{equation}\label{z5}
    \abs{ \int_{E^{(1)}}\phi_m \div V - \int_{F^{(1)}}\phi_m \div V }
    \le 
    | \div V |((E^{(1)}\cup F^{(1)})\cap B_{r_m}(x))
    \le \eps \frac{\meas(B_{r_m}(x))}{r_m} \, ,
    \end{equation}
    for $r_m$ small enough, as a consequence of (ii).
    
    Recalling that $|\nabla \phi_m| \le C(N) r_m^{-1}$, we get
    \begin{equation}\label{z7}
    	\abs{ \int_E \nabla\phi_m \cdot V\di\meas -  \int_F \nabla\phi_m\cdot V\di\meas }
    	\le C(N) \norm{V}_{L^\infty}   \frac{\meas(E\Delta F \cap B_{r_m}(x))}{r_m}
    	\le \eps \frac{\meas(B_{r_m}(x))}{r_m} \, ,
    \end{equation}
    for $r_m$ small enough, where in the last inequality we used \autoref{prop:blowupsandharmcoord} and $\nu_E(x) = \nu_F(x)$ to infer that
    \begin{equation*}
    	\lim_{m\to \infty} \frac{\meas(E\Delta F \cap B_{r_m}(x))}{\meas(B_{r_m}(x))}
    	= 	
    	\lim_{m\to \infty} \meas_m(E_m\Delta F_m\cap B_1^m(x))
    	=
    	\Leb^n(H_{\nu_E(x)} \Delta H_{\nu_F(x)}\cap B_1(0^n)) = 0 \, .
    \end{equation*}
    Subtracting \eqref{z8} and \eqref{z9}, taking into account \eqref{z5} and \eqref{z7} we deduce
    \begin{equation}\label{z10}
    	\limsup_{m\to \infty} \abs{\frac{r_m}{B_{r_m}(x)}\int_{\mathcal{F}E}\phi_m \left(V\cdot\nu_E\right)_{\mathrm{int}}\di\Per_E
    	-
        \frac{r_m}{B_{r_m}(x)}\int_{\mathcal{F}F}\phi_m\left(V\cdot\nu_F\right)_{\mathrm{int}}\di\Per_F}
        \le 2\eps \, .
    \end{equation}
    On the other hand, by (i),
    \begin{align*}
       &\,\limsup_{m\to \infty} 	\abs{\frac{r_m}{B_{r_m}(x)}\int_{\mathcal{F}E}\phi_m \left(V\cdot\nu_E\right)_{\mathrm{int}}\di\Per_E
       	- \left(V\cdot\nu_E\right)_{\mathrm{int}}(x) \frac{r_m}{B_{r_m}(x)}\int_{\mathcal{F}E}\phi_m \di\Per_E}
       	\\ \quad \le &\,
       	\limsup_{m\to \infty} C(N) \frac{\meas(B_{r_m}(x))}{r_m \Per_E(B_{r_m}(x))} \fint_{B_{r_m}(x)} \abs{\left(V\cdot\nu_E\right)_{\mathrm{int}}(y) - \left(V\cdot\nu_E\right)_{\mathrm{int}}(x)} \di \Per_E(y)\, =\, 0\, .
    \end{align*}
    Finally, observe that
    \[
    	\lim_{m\to \infty} \frac{r_m}{B_{r_m}(x)}\int_{\mathcal{F}E}\phi_m \di\Per_E
    	=
    	\lim_{m\to \infty} \fint \phi_m \di \Per_{E_m}
    	=
    	\fint_{H_{\nu_E(x)}} \phi \di \haus^{n-1} \, , 
    \]
    which along with \eqref{z10} and the identity $\nu_{E}(x) = \nu_F(x)$ allows to conclude that
    \[
    	\abs{ \left(V\cdot\nu_E\right)_{\mathrm{int}}(x) - \left(V\cdot\nu_F\right)_{\mathrm{int}}(x) }
    	\fint_{H_{\nu_E(x)}} \phi \di \haus^{n-1}
    	\le 2\eps \, .
    \]
    Since $\eps>0$ is arbitrary the proof of \eqref{zz0} is complete.

    \bigskip
    
    Let us now pass to the proof of \eqref{zz1} and \eqref{zz2}.
    To simplify our notation we assume that $\chi_E$ and $\chi_F$ are pointwise defined as $\chi_E:= \chi_{E^{(1)}} + \frac{1}{2}\chi_{E^{(1/2)}}$ and $\chi_F:= \chi_{F^{(1)}} + \frac{1}{2}\chi_{F^{(1/2)}}$.

    Thanks to \cite[Theorem 5.3]{BuffaComiMiranda19} we know that $V_E:= V\chi_E\in\mathcal{DM}^{\infty}(X)$, for any
    $V\in \mathcal{DM}^{\infty}(X)$, hence as a consequence of \cite[Theorem 6.20]{BuffaComiMiranda19} it holds
    \begin{subequations}\begin{align}\label{z3}
    	\div( V \chi_{E\cap F})& = \div(V_E \chi_F)
    	=
    	\chi_{F^{(1)}} \div(V \chi_F) + (V_E\cdot\nu_F)_{{\mathrm{int}}} |D\chi_F| \, ,\\
\label{z6}
    	\div(V\chi_E)& = \chi_{E^{(1)}} \div V + (V\cdot\nu_E)_{{\mathrm{int}}}|D\chi_E| \, .
    \end{align}\end{subequations}
    In particular, applying again \cite[Theorem 6.20]{BuffaComiMiranda19}, we deduce that
    \[\begin{split}
    	(V\cdot\nu_{E\cap F})_{{\mathrm{int}}} |D\chi_{E\cap F}|
    	&=
    	\div( V \chi_{E\cap F}) - \chi_{(E\cap F)^{(1)}} \div V
    	\\& =
    	\chi_{F^{(1)}} (V\cdot\nu_E)_{{\mathrm{int}}}|D\chi_E| + (V_E\cdot\nu_F)_{{\mathrm{int}}} |D\chi_F|\, ,
    \end{split}\]
    which restricted to  $F^{(1)}$ implies \eqref{zz1}. Arguing symmetrically we get \eqref{zz2}.

    \bigskip

    Let us finally prove \eqref{zz3}. Thanks to \autoref{thm:cutandpaste} we know that $\{\nu_{E\cap F} = \nu_E\} = \{ \nu_E = \nu_F \}$ up to a $\haus^h$-negligible set, hence \eqref{zz3} follows from \eqref{zz0} applied to $E$ and $E\cap F$.
\end{proof}

\section{An example}\label{sec:example}

In this last section we discuss a class of sets of finite perimeter and essentially bounded divergence measure vector fields to which the theory developed before can be applied: level sets of distance-type functions. In particular, \autoref{prop:leveld} below applies to distance functions from closed sets, sufficiently far away from the closed set itself, as it follows from the Laplace comparison theorem, see \cite{Gigli15,CavallettiMondino20}. Similar results have been obtained for the distance from a given point in \cite[Proposition 2.30]{AntonelliPasqualettoPozzetta21}.

The interest towards these examples comes from some recent applications, where level sets of distance-type functions have been used to construct variations of sets of finite perimeter solving variational problems: in \cite{AntonelliPasqualettoPozzetta21} variations via balls (that are sublevel sets of distances from points) have been used in the study of the isoperimetric problem on $\RCD(K,N)$ spaces, while in \cite{MondinoSemola21} \autoref{prop:leveld} is exploited to prove Laplacian bounds for the distance function from the boundary of locally perimeter minimizing sets of finite perimeter (corresponding to vanishing of the mean curvature in the smooth framework) in the same setting.

\begin{proposition}\label{prop:leveld}
Let $(X,\dist,\meas)$ be an $\RCD(K,N)$ metric measure space. Let $\Omega\Subset\Omega'\subset X$ be open domains and let $\phi:\Omega'\to\setR$ be a $1$-Lipschitz function such that 
\begin{itemize}
\item[i)] $\abs{\nabla \phi}=1$ $\meas$-a.e.\ on $\Omega'$;
\item[ii)] there exists $L\le 0$ such that $\Delta\phi \ge L$ in the sense of distributions on $\Omega'$. In particular, $\phi$ has measure valued Laplacian on $\Omega'$.
\end{itemize}
Then, for $\Leb^1$-a.e.\ $t$ such that $\{\phi=t\}\cap \Omega\neq\emptyset$ it holds that $\{\phi<t\}$ is a set of locally finite perimeter in $\Omega$ and 
\[
\left(\nabla\phi\cdot\nu_{\{\phi<t\}}\right)_{\mathrm{int}}=\left(\nabla\phi\cdot\nu_{\{\phi<t\}}\right)_{\mathrm{ext}}=-1\, \quad \Per_{\{\phi<t\}}\text{-a.e.\ in $\Omega$}\, .
\]
\end{proposition}

\begin{proof}
By the coarea formula $\{\phi<t\}$ is a set of finite perimeter for $\Leb^1$-a.e.\ $t$. 

Since $\phi$ has measure valued Laplacian on $\Omega'$ and $\abs{\nabla \phi}=1$ $\meas$-a.e., $\nabla\phi$ is a bounded vector field with measure valued divergence on $\Omega'$, i.e.\ $\nabla\phi\in \mathcal{DM}^{\infty}(\Omega')$.\\
Moreover, for $\Leb^1$-a.e.\ $t$, $\abs{\Delta \phi}(\{\phi=t\})=\abs{\div \, \nabla \phi}(\{\phi=t\})=0$. For any such $t$, thanks to \cite[Theorem 6.20]{BuffaComiMiranda19} we have that 
\[
\left(\nabla\phi\cdot\nu_{\{\phi<t\}}\right)_{\mathrm{int}}=\left(\nabla\phi\cdot\nu_{\{\phi<t\}}\right)_{\mathrm{ext}}\, \quad \Per_{\{\phi<t\}}\text{-a.e.}\, .
\]
Therefore it is sufficient to prove that
\begin{equation}\label{eq:trace1}
\left(\nabla\phi\cdot\nu_{\{\phi<t\}}\right)_{\mathrm{int}}=-1\, \quad \Per_{\{\phi<t\}}\text{-a.e.}\, .
\end{equation}
We claim that for $\Leb^1$-a.e.\ $t$ and for $ \Per_{\{\phi<t\}}$-a.e.\ $x\in\{\phi=t\}$ the following hold:
\begin{itemize}
\item[i)] $x$ is a regular point of $(X,\dist,\meas)$ i.e.\ the unique tangent cone of $(X,\dist,\meas)$ is $\setR^n$, for some $1\le n\le N$;
\item[ii)] $x$ is a regular reduced boundary point of the set of finite perimeter $\{\phi<t\}$, i.e.\ any blow-up of $\{\phi<t\}$ at $x$ in the sense of sets of finite perimeter (see \autoref{def:blowupset}) is a half-space in $\setR^n$;
\item[iii)] $x$ is a regular point for $\phi$, i.e.\ any blow-up of the function $\phi$ at $x$ is a linear function $\phi_x:\setR^n\to\setR$;
\item[iv)] if $\mathbb{H}^n\subset\setR^n$ is the blow-up of $\{\phi<t\}$ at $x$, then $\phi_x=\dist^{\pm}_{\mathbb{H}^n}$ is the signed distance function from $\mathbb{H}^n$.
\end{itemize}
We claim that the conditions (i)--(iv) are sufficient to prove that 
\begin{equation}\label{eq:densitynormaltrace}
\lim_{r\downarrow 0} \frac{\left(\nabla\phi\cdot\nu_{\{\phi<t\}}\right)_{\mathrm{int}}\Per_{\{\phi<t\}}(B_r(x))}{\Per(B_r(x))}=-1\, .
\end{equation}

This is the consequence of a blow-up argument, similar to the one in \cite[Theorem 7.4]{BrueNaberSemola20}. By the Gauss-Green integration by parts formulae, the interior normal trace is the density w.r.t.\ the perimeter measure of the divergence of the vector field $\chi_{\{\phi<t\}}\nabla \phi$, which equals $\Delta \left(\phi-t\right)_{-}$, here $\left(\phi-t\right)_{-}$ is the negative part of $\phi-t$. Under the assumptions (i) -- (iv), after blow-up the set of finite perimeter converges to a half-space, and the function $\phi$ converges to the signed distance from its boundary. In that limit case it is easily verified that  
\[
\left(\nabla\dist^{\pm}_{\mathbb{H}^n}\cdot\nu_{\mathbb{H}^n}\right)_{\mathrm{int}}=-1\, ,
\]
by direct computation.\\ 
Scaling and stability of the distributional Laplacian allow then to infer \eqref{eq:densitynormaltrace}.
\medskip

Let us prove now that $\Leb^1$-a.e.\ $t$ and $\Per_{\{\phi<t\}}$-a.e.\ $x$ verify the properties (i) -- (iv). Since $\abs{\nabla \phi}=1$ $\meas$-a.e.\ on $\Omega'$, by the coarea formula \autoref{thm:coarea} this is equivalent to ask that (i) -- (iv) are verified $\meas$-a.e.\ on $\Omega'$.

The properties (i) and (ii) follow from the general theory of sets of finite perimeter over $\RCD(K,N)$ spaces, since for any sublevel set $\{\phi<t\}$ which has finite perimeter it holds that $\Per_{\{\phi<t\}}$-a.e.\ $x$ is a regular reduced boundary point, where all tangents are Euclidean half-spaces, see \autoref{thm:DGRCD}.
\medskip

It remains to prove (iii) and (iv). In order to do so we first point out that for $\meas$-a.e.\ point any blow-up $\phi_x$ of $\phi$ at $x$ is a linear coordinate function on the tangent $\setR^n$. The argument is classical, and originally due to Cheeger \cite{Cheeger99}, so we avoid it. Observe that the assumption that $\abs{\nabla \phi}=1$ $\meas$-a.e.\ on $\Omega'$ here guarantees that the blow-up is a linear function with slope equal to $1$. Indeed
\[
\fint_{B_r(x)}\abs{\nabla \phi}^2\di\meas=1\, ,\quad\text{for any $0<r<\dist(x,\partial \Omega)$}\, .
\]
Therefore, $\phi_x:\setR^n\to\setR$, that is obtained as limit of $\left(\phi-\phi(x)\right)/r_i$ along the sequence of scaled spaces $(X,\dist/r_i,\meas_x^{r_i},x)$, is a harmonic function such that
\[
\fint_{B^{\setR^n}_{r}(0)}\abs{\nabla \phi_x}^2\di\haus^n=1\, ,\quad\text{for any $r>0$}\, ,
\]
hence it is an affine coordinate with slope $1$.
See also \cite[Theorem 5.4]{Ambrosioetalembedding} for a proof of the almost everywhere harmonicity of blow-ups of harmonic functions more tailored to the setting of $\RCD(K,N)$ metric measure spaces.
\medskip

We are left to verify that the blow-up of $\phi$, that is a coordinate function, hence it is the signed distance function from a certain hyperplane, is actually the signed distance function from the boundary of the blow-up of the set of finite perimeter.\\ 
In order to do so, observe that $\left(\phi-t\right)_{-}\le 0$ and 
\[
\int_{\{\phi>t\}\cap B_r(x)} \left(\phi-t\right)_{-}\di\meas =0\, ,
\]
for any $x\in \{\phi=t\}$.\\
This property is stable under blow-up of the set of finite perimeter and of the function $\phi$, by the $L^1_{\loc}$ convergence of the indicator functions of the scaled sets of finite perimeter. Hence, also $\left(\phi_x\right)_{-}$ vanishes everywhere on the complement of the blow up of $\{\phi<t\}$ at $x$.\\
Since the blow-up of $\{\phi<t\}$ at $x$ is a half-space and the blow-up of $\phi$ at $x$ is an affine coordinate function, the only possibility compatible with the observation above is that $\phi$ blows-up to the signed distance function from the boundary of the blow-up to $\{\phi<t\}$ at $x$, as we claimed.
\medskip

To conclude, observe that by \eqref{eq:densitynormaltrace} the density of the interior normal trace of $\nabla \phi$ on the boundary of $\{\phi<t\}$ at $x$ is $-1$ for $\Per_{\{\phi<t\}}$-a.e.\ $x$. This is sufficient to prove \eqref{eq:trace1} by the Lebesgue differentiation theorem, since the perimeter measure is asymptotically doubling, see \autoref{prop:asymdoub}.
\end{proof}

\end{document}